\documentclass[10pt]{article}
\usepackage{amsmath, amsthm,amsfonts,amssymb,url,color,epsfig,graphics}
\usepackage{chngcntr}
\usepackage{apptools}

\usepackage{comment}
\usepackage{enumerate}
\usepackage{graphicx}
\usepackage[hang,small]{caption}
\usepackage[margin=1in]{geometry}

\newtheorem{thm}{Theorem}[section]
\newtheorem{Prop}[thm]{Proposition}
\newtheorem{lem}[thm]{Lemma}

 \newtheorem{Rmk}[thm]{Remark}
 \newtheorem{Lem}[thm]{Lemma}
  \newtheorem{Def}[thm]{Definition}
 \definecolor{blu}{rgb}{0.0, 0.28, 0.67}
 \definecolor{orange}{rgb}{0.89, 0.26, 0.2}
   
\def\R {\mathbb{R}}

\def\de {{\partial}}
\def\eps {{\varepsilon}}

\newcommand{\Div}{\operatorname{div}}
\newcommand{\rot}{\operatorname{rot}}

\newcommand{\ba}{\begin{aligned}}
\newcommand{\ea}{\end{aligned}}
\newcommand{\be}{\begin{equation}}
\newcommand{\ee}{\end{equation}}

\providecommand{\keywords}[1]{\textbf{\textit{Keywords:}} #1}

\numberwithin{equation}{section}

\begin{document}

\title{Nonresonant Bilinear Forms \\for partially dissipative hyperbolic systems\\ violating the Shizuta-Kawashima condition}
\author{Roberta Bianchini 
 \footnote{IAC, Consiglio Nazionale delle Ricerche, via dei Taurini 19, 00185 Rome (Italy)  {\color{blue} {\tt roberta.bianchini@cnr.it}}}, 
Roberto Natalini
 \footnote{IAC, Consiglio Nazionale delle Ricerche, via dei Taurini 19, 00185 Rome (Italy) {\color{blue} {\tt roberto.natalini@cnr.it}}}
}
\date{}

\maketitle 
\begin{abstract} 
In the context of hyperbolic systems of balance laws, the Shizuta-Kawashima coupling condition guarantees that all the variables of the system are dissipative even though the system is not totally dissipative. Hence it plays a crucial role in terms of sufficient conditions for the global in time existence of classical solutions. However, it is easy to find physically based models that do not satisfy this condition, especially in several space dimensions. In this paper, we consider two simple examples of partially dissipative hyperbolic systems violating the Shizuta-Kawashima condition ([SK]) in 3D, such that some eigendirections do not exhibit dissipation at all. We prove that, if the source term is non resonant (in a suitable sense) in the direction where dissipation does not play any role, then the formation of singularities is prevented, despite the lack of dissipation, and the smooth solutions exist globally in time. The main idea of the proof is to couple Green function estimates for weakly dissipative hyperbolic systems with the space-time resonance analysis for dispersive equations introduced by Germain, Masmoudi and Shatah. More precisely, the partially dissipative hyperbolic systems violating [SK] are endowed, in the non-dissipative directions, with a special structure of the nonlinearity, the so-called Nonresonant Bilinear Form for the wave equation (see Pusateri and Shatah, CPAM 2013).
\end{abstract}


\section{Introduction}\label{sec:intro}
The Cauchy problem for hyperbolic balance laws in several space dimensions takes the form

\begin{align}
& \partial_t u + \sum_{i=1}^n \partial_{x_i} F_i(u)=G(u), \label{cauchy}\\
& u(x,0)=u_0(x).\label{initial}
\end{align}
Here the unknown is a vector $u=(u_1, \cdots, u_k) \in\R^k$, the fluxes are given by smooth functions $F=(F_1,\dots,F_n)$, and the source term $G(u) \in \R^k$ is also smooth.

Several fundamental laws of physics enter this framework: for the relevance
of hyperbolic conservation laws in continuum physics we refer to \cite{whitham, majda, dafermos}. 
A peculiar feature of nonlinear hyperbolic systems is the possible loss of regularity. Even with smooth initial data, it is well-known that the solution can develop shocks in finite time. As a matter of fact, in many cases  solutions to a general system \eqref{cauchy}--\eqref{initial}  fail to be smooth, because some of their components or their gradients blow up pointwise in finite time. Explicit examples of this phenomenon are illustrated for instance in
\cite{majda, John2}. Nevertheless, there are some specific classes of hyperbolic systems, whose solutions stay smooth for all times, at least for small initial data. The simplest physical mechanism preventing the formation of singularities is dissipation, one of the more relevant hyperbolic dissipative models with global smooth solutions being 
the compressible Euler system with damping \cite{nishida,sideris}. 
More general cases have been considered:
in \cite{HN, yong, SK} a rigorous framework was proposed to characterize a quite general class of partially dissipative hyperbolic systems whose smooth solutions are global in time. We try here to explain the guidelines of the development of this theory. We start with a physically reasonable class of systems endowed with a convex entropy $\eta(u)$ and related entropy flux $\mathcal{G}(u)=(\mathcal{G}_j(u))_{j=1, \cdots,n}$, such that
\begin{align*}
\eta=\eta(u), \quad \text{with} \quad  \eta'(u) \cdot F_j'(u)=\mathcal{G}'_j(u), \quad j=1, \cdots, n.
\end{align*}
The simplest preliminary condition that one can require is \emph{entropy dissipation}. We remind to \cite{chen, HN} and references therein for an exhaustive discussion. This condition can be stated as follows. Consider an equilibrium value $\bar u$ for system \eqref{cauchy}, i.e. such that $G(\bar u)=0$. For every $u$ in a given domain $\Omega \subset \R^k$, the following inequality has to be satisfied:
\begin{equation}\label{entropy}
(\eta'(u) -\eta'(\bar u) ) \cdot G(u)\leq 0.
\end{equation}
Unfortunately, even in very simple cases, it is possible to see that condition \eqref{entropy} alone is not enough to prevent the appearence of singularities in finite time. Therefore, an additional assumption is needed. Hence, together with entropy dissipation \eqref{entropy}, one requires the so-called Shizuta-Kawashima condition, which has been originally introduced in \cite{SK} and can be presented as follows:\\

[SK] {\it Shizuta-Kawashima}. Given an equilibrium value $\bar u$ for system \eqref{cauchy}, if $z\in Ker G'(\bar u)$, then there is no vector $\xi=(\xi_1,\cdots,\xi_n)\in \R^n-\{0\}$ such that $z$ is an eigenvector of the matrix $\sum_{i=1}^n  F'_i(\bar u)\xi_i$.\\

This condition can be read in many different ways. In terms of stability,
it guarantees the necessary coupling between conserved/non conserved quantities
in order to have dissipation in each state variable. Hence, the [SK] coupling assures that the amount of dissipation provided by the entropy inequality \eqref{entropy} is distributed in all the directions, and this framework is \emph{sufficient} to guarantee the well-posedness of classical solutions to \eqref{cauchy}-\eqref{initial} for all times.
In \cite{HN, yong} it is indeed proved that, under both entropy dissipation and the Shizuta-Kawashima conditions, and for small initial data, problem \eqref{cauchy}-\eqref{initial} has global in time smooth solutions.
However, there are many physical
systems, especially in the multidimensional case, which have dissipative entropies as in \eqref{entropy}, but do not satisfy the coupling condition [SK]. An interesting example is the
model for gas dynamics  in rotational and vibrational non-equilibrium in Eulerian coordinates in $n$-space dimensions, which reads
\begin{equation}\label{gasdynamics}
\begin{cases}
\partial_t \rho + \Div (\rho u)=0,\\
\partial_t (\rho u) + \Div (\rho u\otimes u) +\nabla P(\rho,e)=0,\\
\partial_t \left( \rho (\frac 1 2 |u|^2 + e + q)\right) + \Div \left( (\frac 1 2 \rho |u|^2 + \rho e + \rho q + P(\rho,e))u\right)=0,\\
\partial_t (\rho q)+ \Div (\rho q u)=\frac{1}{\tau} \left(\rho(Q(e)-q)\right),
\end{cases}
\end{equation}
where $\rho$, $u$, $e$, and $q$ are respectively the density, the velocity, the internal energy and the internal vibrational energy of the gas, and $\tau$ is the relaxation time for the rotational and vibrational modes, see for instance \cite{whitham} for further details on the modeling. The last equation models the fact that, during rapid changes in the flow, the internal energy $e$ may lag behind the equilibrium value corresponding to the ambient pressure and density. On the other hand, the translational energy reaches the equilibrium faster than the rotational and vibrational energies. It is easy to show that this system does not satisfy the Shizuta-Kawashima condition in all space dimensions (see for instance \cite{HN} for the 1D case). Nevertheless, in the one-dimensional case, Y. Zeng in \cite{zeng} was able to obtain a global existence result using Lagrangian coordinates, while the multidimensional case is still open.
Another interesting physical model which enters the class of dissipative hyperbolic systems violating [SK] is represented by the nonlinear Maxwell equations
\begin{equation}\label{KD}
\begin{cases}
\partial_t D - \rot H= 0,\\
\partial_t B + \rot E= 0, \\
\Div D=\Div B=0,
\end{cases}
\end{equation}
where $E$ and $H$ are the electric and magnetic field respectively, and $D$ and $B$ are the related displacements, given by
$$ B=\nu_0H,\ D=\varepsilon_0 (1+\chi)E  , $$
with $\varepsilon_0\chi E$ which stands for a polarization term. In the so-called Kerr-Debye model, the term $\chi$ is given by a relaxation equation: 
\begin{equation}\label{pol}
\partial_t \chi=\frac{1}{\tau} \left(\varepsilon_r |E|^2 - \chi\right),
\end{equation}
{where $\eps_r>0$ is the dielectric constant of the Maxwell equations.}
For further details on the modeling aspects, see for instance \cite{zio}. It is shown in \cite{HN}, see also \cite{CH, CHN} and references therein, that when the electric field vanishes, the [SK] condition is no longer satisfied. Nevertheless, in \cite{CHN} global existence of smooth solutions is proved for all small initial data. 

Finally, we mention two other relevant examples of systems violating [SK]: the hierarchy of equations coming from the kinetic
formulation of multi-branch entropy solutions of scalar conservation laws \cite{brenier,BCN}, and a very simple example on traffic flow presented in \cite{li}.

\paragraph{{Some attempts to go beyond the Shizuta-Kawashima condition.}}
Although several physical models violate the Shizuta-Kawashima condition [SK], at the present time there is no general theory which allows to go beyond [SK] still providing global in time existence of smooth solutions to multidimensional dissipative hyperbolic systems \eqref{cauchy}. This condition turns out to be indeed too restrictive and only some special cases, of which \cite{zeng} is the first example, have been treated so far.
To overcome this problem, one first direction is to weaken the requirements on the linearized part of the system, as started in \cite{BZ}. 
The refined linear analysis of \cite{BZ} applies to systems for which [SK] fails on sets of the frequency space of zero Lebesgue measure. Although being strictly weaker than [SK], this requirement still enables a certain decay rate for every variable of the linear system. This slower linear decay is then employed by the authors of \cite{BZ} to deal with a special class of nonlinear systems. {Among the examples of nonlinear partially dissipative system for which global existence of smooth solutions for small data has been achieved, we mention the regularity-loss type equations and in particular the Timoshenko system, see \cite{timo}.} However, this is again a special class of equations.
The (nonlinear) framework and also the point of view of the present manuscript is completely different: in particular, we deal with models in which some of the variables are not dissipated at all and this lack of dissipation persists in any point of the frequency space. Since there is not any dissipation in some directions, a key role is played here by the structure of the nonlinear term.
The strategy is to bypass the failure of the [SK] condition on the linearized part of the system, by taking advantage of some special features of the nonlinear terms. We recall indeed that an unsolved conjecture by Majda (see \cite{majda}) states that \emph{the Cauchy problem for the symmetric hyperbolic system with totally linearly degenerate characteristics admits a global, classical solution unless the solution itself blows up in ﬁnite time when the initial data belong to Sobolev spaces $H^s$ for $s \ge  1+\frac n 2$.}
Our idea is to rely on this property for partially dissipative hyperbolic systems, and in particular for the directions where there is no dissipation. The first significant obstacle to overcome is to find a framework which allows to unify the (dissipative) contributions from some eigendirections and the good nonlinear structure (in terms of linear degeneracy) of the complementing directions. 

 A first attempt to implement this approach in a very specific case is contained in \cite{MN}, where, for a simple class of one-dimensional hyperbolic systems, \emph{linear degeneracy counterbalances the lack of dissipation in preventing shocks}. Here we try to make a step forward in this last direction. We consider a partially dissipative hyperbolic system violating [SK] in some directions (and for every point of the frequency space), where the equations for the non-dissipative variables have a (nonlinear) source term with a special structure, i.e. a suitable generalization of linear degeneracy, which prevents  singularity formation even though in some directions there is no dissipation at all.

As a strong motivation for our work, the latter is a natural framework for many interesting physical models. We refer again to 
systems \eqref{gasdynamics} and \eqref{KD}, and to the class of equations considered in \cite{MN}: all these are examples where linear degeneracy occurs only at the points where [SK] fails and, at the same time, global in time smooth solutions in 1D exist. As a matter of fact, physically interesting models can exhibit a special structure of the nonlinearity which is able to prevent singularity formation: it is then worth investigating in this direction, trying to provide a more general classification of the ``good'' nonlinear terms and taking advantage of them.

We now need a step back to provide a brief treatment of nonlinear equations with special nonlinear non-dissipative terms, such that no breakdown occurs.

\paragraph{{Global smooth solutions for small data to dispersive equations: the role of resonances and null structures.}}
The investigation on global existence of smooth solutions to multidimensional hyperbolic, and more generally dispersive, equations with quadratic source terms and small initial data takes its roots around the Eighties. A deep historical and mathematical survey on the topic can be found in \cite{Lannes}. A general common feature of a series of paper is that the linear dispersive terms in the considered equations tend to force the solution to spread and to decay, and the contribution of the nonlinear terms is controlled by dispersion. Since dispersion increases with space dimension, a first class of global existence results has been obtained in dimension $n=4$ by S. Klainerman \cite{Klainerman0}. As showed by F. John, \cite{John}, in lower space dimensions the nonlinearity can lead to blow up in finite time for arbitrarily small data. In this case, a precise structure of the nonlinearity, the so-called \emph{null condition}, introduced by Klainerman \cite{Klainerman1} and Christodoulou \cite{C}, prevents the formation of singularities, at least in the case of systems of nonlinear wave equations and some related examples. \\
More recently, a crucial contribution to extend the notion of null forms for the wave equation, in the framework of the space-time resonance analysis introduced by Germain, Masmoudi and Shatah, was given by Pusateri and Shatah, see \cite{Pusateri}. The main idea of the space-time resonance method is the following: instead of focusing only on time resonances (in other words, instead of dealing with plane waves), one looks also at spatial localization (wave packets) by using space-weighted estimates. A highlighting survey paper is due to Germain, \cite{Germain}. Normal form methods deal with time resonances, while null forms guarantee spatial localization, so preventing space resonances. Each one tells only a part of the story. A coupling of these two points of view is what is needed to complete the picture.\\
In the context of long-time behavior for the wave equation with space-time resonance techniques, we refer to a recent paper by Pusateri and Shatah, \cite{Pusateri}. Notice that, quadratic source terms satisfying the null condition for the wave equation are actually equivalent to the compatible forms for hyperbolic systems, \cite{HJ}, which are the ones having the weakly sequential continuity described by compensated compactness, \cite{Tartar1}. A preliminary result for hyperbolic systems is proved in \cite{GLZ}. Actually, the first result in this direction is due to an unpublished paper by Tartar \cite{Tartar2}, where he proved global existence of smooth solutions in 1D to a semilinear hyperbolic system with quadratic source satisfying a non-crossing condition for the characteristics (which is actually equivalent to the null condition for this particular case). This paper has been recently revisited in \cite{BG}.
We only recall that the classical example of wave equation with null condition, in 1D for simplicity, 
$$\de_{tt} u - c^2\de_{xx}u=(\de_tu)^2-(\de_xu)^2,$$
can be rewritten as the first order hyperbolic system
\begin{align}\label{eq:tartar}
\de_t u_1 + c \de_x u_1&=u_1u_2, \notag\\
\de_t u_2 - c \de_x u_2&=u_1u_2,
\end{align}
which is the simplest one among the semilinear hyperbolic systems satisfying the Tartar non-crossing condition and covered by his result \cite{Tartar1}. 

\paragraph{{Linear degeneracy and null structures in 1D: John's formula.}}
As already pointed out in \cite{John2} (see also \cite{sideris2, ferapontov, CHN}), linearly degenerate eigenvalues are strongly connected with the null condition. A clear indication of that is provided for instance by the systems
\eqref{gasdynamics} and \eqref{KD} discussed before. These two examples have eigenvalues which are linearly degenerate in the points where [SK] is not satisfied and their smooth solutions are well behaved for all times in 1D. Moreover, at least in the one-dimensional case this connection between \emph{linear degeneracy} and \emph{null forms} appears to be clear by means of \emph{John's decomposition formula}, \cite{John}. Consider a quasi-linear hyperbolic system in one space dimension,
$$\de_t u + A(u) \de_x u=0.$$
Thanks to John \cite{John}, one can derive the following decomposition
\begin{align}\label{eq:johndec}
\de_t p_i + \lambda_i(u) \de_x p_i = - \sum_{k=1}^N p_i p_k \lambda_i'(u) r_k(u) + \sum_{j, k=1}^N(\lambda_k(u)-\lambda_j(u))p_j p_k^t l_i(u) r_j'(u)r_k(u), \quad p_j=(l_j)^T \de_x u,
\end{align}
where $\lambda_i(u)$ are the eigenvalues of $A(u)$ and $l_j, r_j$ are the related left and right eigenvectors. Now, if the considered system is \emph{totally linearly degenerate} (TLD), then \emph{there is no diagonal term in the  above quadratic form, namely there is no term of the form $p_k^2$}. This observation is very well explained in \cite{CH}. It is really a generalization of the Tartar non-crossing condition \cite{Tartar2}, which was stated for semilinear systems, to the quasilinear case. The equivalence with the null condition is showed for instance in \cite{BG} for the classical example of the wave equation. In the 1D case, it provides a recipe to extend the toolbox on null forms which has been developed for the wave equation (see \cite{John, Klainerman1, Pusateri} and references therein) to the case of hyperbolic systems. In general, this is far from being obvious: as pointed out by John in the context of almost global solutions for the wave equation with null structure in \cite{John} ``we have heavily relied on the invariance properties of the operator $\de_{tt}-\de_{xx}$. This makes it difficult to generalize the result to other quasilinear equations that linearize to a homogeneous hyperbolic system''. Of course there are special cases where the specific structure of the considered hyperbolic system simplifies the analysis, like the case of diagonal systems, and this includes for instance some incompressible fluid-dynamics models, and in particular it is known that compressible and irrotational flows can be expressed as a second-order quasilinear wave equation with respect to the velocity potential, see \cite{BSV} and references therein. 

\paragraph{{Our approach: the interplay between partial dissipation and null structures.}}
Even though a general theory is still missing in the multiD case, the properties of quadratic null structures, which are well-known for the wave equation, can be translated in the quite general framework of hyperbolic systems in 1D and at least in some special cases in 2D and 3D, see \cite{Crist}. Therefore, this gives hope that a similar property could persist for multiD hyperbolic systems, where the theory of global smooth solutions for small data in the general case is based on the role of dissipation (the [SK] condition, \cite{BHN}). In other words,  this strongly suggests that a more careful study of the interplay between dissipation and linearly degenerate fields, and more generally the null condition, could give at least some deeper insights on the long-time behavior of multidimensional hyperbolic systems.
It is also worth pointing out that a specific (and in general weaker) form of null structure has been recently found for the 3D compressible Euler equations by Luk \& Speck \cite{speck}. That result has already been implemented in \cite{wei} to prove the global in time existence for small data of the 3D compressible irrotational Euler equations for Chaplygin gases, whose nonlinear term satisfies the classical Klainerman null condition \cite{Klainerman1}, provided that the system is rewritten in terms of two coupled quasilinear wave equations (actually they are geometric wave equations, as they rely on the geometric wave operator). 
Using that null structure for irrotational Chaplygin gases, in light of our preliminary study carried out in the course of this paper, it is reasonable to conjecture that for instance irrotational Chaplygin gases out of rotational and vibrational equilibrium (system \eqref{gasdynamics} in the irrotational case for Chaplygin gases) admits global in time smooth solutions for small data. Theoretically, this is indeed a consequence of our analysis and our results presented below. There is however a non-negligible obstruction, which is represented by the fact that the null property of the 3D compressible Euler equations in that specific setting has been proved by relying on the classical Klainerman framework of commuting vector fields and geometric wave operators \cite{Klainerman1}. In order to exploit the partial dissipation and to couple this property of the \emph{linearized equations} with the null form of the \emph{nonlinear terms} instead, our intuition is that at least the simplest methodology is based on Fourier analysis and the reformulation/extension of Klainerman's method which is known as space-time resonance analysis, which is introduced below. Therefore the analysis of system \eqref{gasdynamics} for irrotational Chaplygin gases can be carried out once the null structure has been translated in Fourier space, without using vector fields and geometric wave operators.  Although being a strong motivation, the extension of our method to that system is out of the scope of the present paper and will be investigated in a future work. {The Kerr-Debye system \eqref{KD} is even closer to the toy model (the second one) that we investigate in this paper 
and its three dimensional version is probably the first example to which our method could be applied. In fact, it is known (see \cite{CH} and references therein) that in 3D, system \eqref{KD} in the variables $(D, H, \chi) \in \mathbb{R}^7$ with $\chi \ge 0$ is a quasilinear symmetrizable hyperbolic system with eigenvalues, for $\xi \neq 0$, given by
$$\lambda_1(\chi, \xi)=\lambda_2 < \lambda_3 = \lambda_4 = \lambda_5 =0 < \lambda_6 = \lambda_7 = - \lambda_1.$$
Moreover, each eigenvalue is linearly degenerate: it is enough to check this property for $\lambda_1$, which gives $\nabla \lambda_1 \cdot r_1=0$, where $r_1$ is the corresponding right eigenvector. Since the linear degeneracy is also valid in 1D, in that case one can apply John's decomposition formula \eqref{eq:johndec} and obtain that there is no any \emph{diagonal} quadratic term (thanks to the linear degeneracy), so that the system satisfies the classical null condition (or Tartar non-crossing condition as in system \eqref{eq:tartar}). In 1D, this is in fact the reasoning on which the proof of global in time existence for small data in \cite{CH} is based. Such a proof cannot be extended directly to the 3D case as the John decomposition formula is not known to be valid in 3D. However, a connection between linear degeneracy of eigenvalues and nonresonant bilinear forms should be valid in the 3D case. 
To motivate this intuition, let us apply the change of variables $\mathcal{D}=\frac{D}{1+\chi}$ to system \eqref{KD} where we assume $\eps_0=\eps_r=\nu_0=1$ for simplicity. Recalling that $|E|^2=\frac{|D|^2}{(1+\chi)^2}$, system \eqref{KD} reads
\begin{align*}
\de_t \mathcal{D}- \frac{\rot H}{1+\chi}&= -\frac{\mathcal{D}}{\tau (1+\chi)}(|\mathcal{D}|^2 - \chi), \\
\de_t H + \rot \mathcal{D}&=0, \\
\de_t \chi&=\frac{1}{\tau}(|\mathcal{D}|^2-\chi), \\
\rm{div} H&=\rm{div} D=0. 
\end{align*}
The first two equations can be written in the form of a wave equation as follows
\begin{align*}
\de_{tt} H + \rot \left(\frac{\rot H}{1+\chi}\right) &=\rot \left( \frac{\mathcal{D}}{\tau (1+\chi)}(|\mathcal{D}|^2 - \chi)\right).
\end{align*}
The [SK] condition is known to fail (\cite{CH}) when $E=0$. Note that as $D=(1+\chi)E$, then $E=0$ implies $D=0$. Therefore the point here is to prove that, when $D=0$, the source term of the following wave equation
\begin{align*}
\de_{tt} H + \frac{\rot^2 H}{1+\chi} = \rot H \cdot\rot \left(\frac{1}{1+\chi}\right)- \frac{\rot}{\tau}\left(\frac{\chi}{1+\chi}\right)
\end{align*}
is nonresonant in the sense of Definition \ref{def:Nonresonant-Bil-Form}. This intuition will be the object of investigation of a future work. 
 }

The main purpose of this paper is to make a step forward in the understanding of the cooperation between linear degeneracy and partial dissipation in the space-time resonance framework introduced in \cite{GMS2} using two toy models with a wave-type structure. 
Although here we are interested in hyperbolic systems of first order, the simple example below should explain why we believe that this is the right setting to investigate our problem. 

\paragraph{{A brief overview on the space-time resonance method.}}
Consider the scalar equation
\begin{equation*}
\partial_t u + i \Lambda u = \mathcal{Q}(u,u), \quad u(t=0)=u_0,
\end{equation*}
where $\Lambda=Op(|\xi|)$ is a real Fourier multiplier and $ \mathcal{Q}(u,u) \in \{ u^2, \bar{u}^2, u\bar{u}\}$. Now define the new unknown function $f(t):=\exp(i \Lambda t) u(t)$. By Duhamel formula, denoting by $\widehat{f}(t,\xi)$ the Fourier transform of $f(t)$, 
$$\widehat{f}(t,\xi)=\widehat{u}_0(\xi) + \int_0^t \int \exp(is \phi(\xi, \eta)) \widehat{f}(s, \xi-\eta) \widehat{f}(s, \eta) \, d\eta \, ds, \quad \phi(\xi, \eta)=|\xi| \pm |\xi-\eta| \pm |\eta|,$$
{where with a slight abuse of notation conjugate terms are omitted.}

From the stationary phase viewpoint, the set $\mathcal{T}=\{(\xi, \eta): \phi(\xi, \eta)=0\}$ is delicate and could create time growths. This is the the first point where damping helps, by {dissipating the amplitudes of the oscillations represented by the phase} $\phi(\xi, \eta)$. Indeed, in the very basic example below,
\begin{equation*}
\partial_t \tilde u + i \Lambda \tilde u + \tilde u= \mathcal{Q}(\tilde u,\tilde u), \quad \tilde u(t=0)=\tilde u_0,
\end{equation*}
the profile is $\tilde f(t)=\exp(i\Lambda t+t) \tilde u(t)$, so that the phase reads $\tilde \phi(\xi, \eta)=|\xi| \pm |\xi-\eta| \pm |\eta|+i$ and $\tilde{\mathcal{T}}=\emptyset $.
Roughly, this is what happens in the case of a (for instance a 2x2) Shizuta-Kawashima system in $u, v$ when $\widehat{u}, \widehat{v} \approx \exp(-t)$ interact with $\widehat{w} \approx \exp(i|\xi| t)$ solution to a transport equation at the level of high frequencies. By contrast, in the low frequency regime, there is no exponential decay in the Shizuta-Kawashima (dissipative) variables and the linear dynamics behave like the heat semigroup $\widehat{u}, \widehat{v} \approx \exp(-|\xi|^2t), |\xi| \exp(-|\xi|^2t)$ (see formula \eqref{dec:U}).  Since we deal with hyperbolic systems of the first order, whose initial data are transported in special directions chosen by the characteristics, for $(\xi, \eta) \in \mathcal{T}$ one can appeal to \emph{spatial localization} to get enough time decay. As widely explained for instance in \cite{GMS1, Pusateri}, in order to take advantage of spatial localization one usually relies on the analysis of the \emph{space resonances} given by the set $\mathcal{S}=\{(\xi, \eta)| \de_\eta \phi(\xi, \eta)=0\}$, integrating by parts in frequency $\eta$ away from $\mathcal{S}$. This is the crucial point where dissipation helps in the low frequency regime. In fact, since in the small frequency regime $\widehat{v}(t, \xi) \approx |\xi| \exp(-|\xi|^2t)$, a reduced version (our model is actually a 3x3 hyperbolic system) of the model that we study roughly behaves in the low-frequency regime as
\begin{align*}
\de_t w + i \Lambda w = v w, \quad 
\de_t v = \Delta v + v^2.
\end{align*}
 Thus we easily see that the time phase associated to the first equation above  is $\phi(\xi, \eta)=|\xi|-|\xi-\eta| + i |\eta|^2$, whose gradient $\de_\eta \phi$ is always far from zero. In other words, the space-resonant set is empty and therefore all the wave-packet interactions become negligible for large times. This just means that the oscillations given by $\exp(i|\xi|t)$ cannot generate any resonance through interactions with the heat semigroup $\exp(-|\xi|^2t)$, because of the fact that the latter simply does not oscillate!

\paragraph{{Presentation and motivations of our models.}}
The toy models that we consider in this paper are crude simplifications of more physically relevant systems, but they capture the most recurrent structure of models coupling the lack of dissipation due to the failure of [SK] with a suitable definition of linear degeneracy (Nonresonant Bilinear Form, see Definition \ref{def:Nonresonant-Bil-Form}).\\
We consider the simplest (and however most representative) form of a partially dissipative hyperbolic system satisfying [SK], coupled with a wave equation in three space dimensions. We describe more precisely the systems. The first one reads as follows:
\begin{equation*}
\begin{cases}
\partial_t u+i \Lambda v = \mathcal{Q}_1(u,v),  \\
\partial_t v+i \Lambda u + v = \mathcal{Q}_2(u,v)+\alpha uw,\\
\partial_t w+i \Lambda w=\mathcal{B}_{nonres}(w,w),
\end{cases}
\end{equation*}
where $\mathcal{Q}_i(u,v), i=1,2$ are generic quadratic terms in $u,v$, $\mathcal{B}_{nonres}(w,w)$ is a Nonresonant Bilinear Form and $\alpha$ is a constant.
Although the dissipation rate of the dissipative Shizuta-Kawashima variables $u,v$ is weakened by the presence of $w$ in the nonlinearity, in this model the amount of dissipation is still enough to prevent the formation of singularities. What happens here, as showed in Theorem \ref{thm:PSK}, is simply that the dissipation of the \emph{linear} equations for $u$ and $v$ is strong enough to handle a quadratic source involving a non-dissipated variable (i.e. the term $uw$). This is just to explicitly show that the dissipation due to the linear operator of the dissipative [SK]-variables $u,v$ is able to handle some specific quadratic sources involving (the non-dissipated variable) $w$, but not a general quadratic source as pointed out in Remark \ref{rmk:nonlinearity}. This is why a new approach is needed. 
In this regard, a more interesting toy model is the following:
\begin{equation*}
\begin{cases}
\partial_t u+i \Lambda v = \mathcal{Q}_1(u,v),  \\
\partial_t v+i \Lambda u + v = \mathcal{Q}_2(u,v),\\
\partial_t w+i \Lambda w=\mathcal{B}_{nonres}(w,w)+\alpha vw,
\end{cases}
\end{equation*}
where the coupling term $vw$ lies in the equation for $w$, whose linear part does not dissipate at all. The last equation with $\alpha=0$ is object of investigation of \cite{Pusateri}. The whole source term $\mathcal{B}_{nonres}(w,w)+\alpha vw$ is not nonresonant and does not fit the framework of \cite{Pusateri}, but encodes some dissipation thanks to $v$ in $vw$.
The natural question is to ask whether the null structure of $\mathcal{B}_{nonres}(w,w)$ and the dissipation hidden in $vw$ are able to force the solutions to stay smooth for all times. This has a positive answer, as Theorem \ref{thm:PSK2} shows, for which the analysis of both the space and time resonances is a key ingredient. The space-time resonance method is then adapted here to non-dispersive equations, i.e. partially dissipative hyperbolic systems of first order and allows to prove that the cooperation of very weak dissipation violating Shizuta-Kawashima and special nonlinear structures is successful for handling quadratic nonlinearities for long times, at least in the class of systems considered here. This shows the potential appeal of our approach for dealing with more concrete and complicated cases, as the examples discussed before.

Even if, at a first glance, our toy models can appear not so relevant from the physical point of view, instead they contain some of the main interesting features of the general case. The simplified scalar version of the 3D nonlinear wave equation was first studied in \cite{Pusateri}. It is considered here to avoid long and technical complications, but the same procedures can be extended to the real (vectorial) 3D wave equation with the null condition. Another reduction, which again simplifies the calculations, is due to the very basic form of the linear part of the partially dissipative hyperbolic system satisfying [SK] that we couple with the scalar wave equation. In terms of the Fourier variables, the linear part of our dissipative system can be merely reduced to the damped scalar wave equation. However, this is a very significant model for [SK]-systems, since the linearized part of these models can be always reduced to a canonical form (i.e., the first two linearized equations of system  \eqref{PK_system_w}) coupling a conservative part with a dissipative one, the so-called \emph{Conservative-Dissipative Form} introduced in \cite{BHN}, by means of a standard change of variables, see Remark \ref{rmk-CD-form} below. Therefore, our system is a simplified but representative model problem, and the same ideas apply to more realistic vectorial systems. The treatment of more (physically) meaningful models, as for instance the multiD compressible Euler equations out from rotational and vibrational equilibrium introduced before, is left for a future investigation.

Finally, we point out that, despite of their relevance, the state of the art on partially dissipative multidimensional hyperbolic systems, even in the [SK] case, is far from being complete and many problems are still open. For instance, in Section \ref{subsec:SK}, just using the decay properties of the Green function of the linearized system, we observe that the [SK]-system \eqref{K_system} can handle any type of quadratic source term even in the equation for the conservative variable $u$, a case which was not considered in \cite{HN, BHN}.

\paragraph{Plan of the paper}
The paper is organized as follows. In Section \ref{sec:first-model} we study a partially dissipative hyperbolic system (of Shizuta-Kawashima type) coupled with the wave equation with nonresonant bilinear source, where the coupling lies in the quadratic source term of the equations for the dissipative variables. Section \ref{sec:second-model} concerns the more interesting case where the coupling between the Shizuta-Kawashima system and the wave equation takes place in the quadratic source of the wave equation, whose linear part does not dissipate at all.

{\paragraph{Notation and conventions}
We highlight the following conventions, which will be used throughout the paper.
\begin{itemize}
\item We use the notation $f_1 \lesssim f_2$ if there exists a constant $C$ such that $f_1 \le C f_2$, where $C$ is independent of time. We denote $f_1\approx f_2$ if $f_1 \le f_2$ and $f_2 \lesssim f_1$.
\item $\mathcal{R}$ denotes the real part and $\mathcal{I}$ the imaginary part.
\item We do not keep track of the $2\pi, i, -1$ factors in front of the integrals, since they do not play any role in the estimates.
\item {We use $\mathcal{F}$ to denote the Fourier transform.}
\item We denote $\Lambda=(-\Delta)^\frac 12$ the operator associated with the Fourier multiplier $|\xi|$.
\item We use the notation $R_j$ for the $j$-th component of the Riesz transform $R$.
\item We occasionally identify operators with their symbols in order to avoid additional notations (this is for instance the case of the operator $\Lambda$). In particular, we denote, with a small abuse of notation,
$$T_{\mu(\xi, \eta)}(f,g)=\mathcal{F}^{-1}\int {\mu(\xi, \eta)}\widehat{f}(\eta-\xi) \widehat{g}(\eta) \, d\eta$$ the operator associated to the symbol $\mu(\xi, \eta)$.
\item We use the notation $\mu_0, \, \mu_s$ to denote any symbol in $\mathcal{B}_0, \, \mathcal{B}_s$ respectively (see the definitions below).
\end{itemize}}

\section{A first model violating the Shizuta-Kawashima condition}\label{sec:first-model}
We start with the simplest example of partially dissipative hyperbolic system violating the Shizuta-Kawashima condition, hereafter denoted by [SK]. In Fourier variables, the linear part of this system is as follows:
\begin{equation}\label{linear-PK-system}
\begin{cases}
\partial_t \widehat{u}+i|\xi| \widehat{v}=0,\\
\partial_t \widehat{v}+i|\xi| \widehat{u}+\widehat{v}=0,\\
\partial_t \widehat{w}+i|\xi| \widehat{w}=0,
\end{cases}
\end{equation}
where $(t, \xi) \in \mathbb{R}^+ \times \mathbb{R}^3$. As we will see in Section \ref{subsec: spectral-analysis}, the variables $u, v$ satisfy [SK], then will be identified as the \emph{dissipative variables}, while $w$ is decoupled and does not dissipate at all. 
We introduce the Bilinear Pseudoproduct Operator
\begin{equation}\label{def-pseudoproduct}
T_{m(\xi, \eta)}(f,g):=\mathcal{F}^{-1} \int m(\xi, \eta) \widehat{f}(\eta) \widehat{g}(\xi-\eta) \, d\eta,
\end{equation}
where 
$m(\xi,\eta)$ is a Nonresonant Bilinear Form, defined just below.

We consider system (\ref{linear-PK-system}) with a quadratic source term, coupling all the three variables:
\begin{equation}
\label{PK_system}
\begin{cases}
\partial_t \widehat{u}+i|\xi| \widehat{v}=a_u \widehat{u^2} + b_u \widehat{v^2} + c_u \widehat{uv},  \\
\partial_t \widehat{v}+i|\xi| \widehat{u}+\widehat{v}=a_v \widehat{u^2} + b_v \widehat{v^2} + c_v \widehat{uv} + d_v {\widehat{uw}},\\
\partial_t \widehat{w}+i|\xi| \widehat{w}=\widehat{T}_{m(\xi, \eta)}(w,w),
\end{cases}
\end{equation}
where $a_u, b_u, c_u, a_v, b_v, c_v, d_v$ are constant values. Following \cite{Germain, Pusateri}, define the \emph{profile} of $w$ \begin{align}\label{def:profile1}
f_w(t):=\exp(i\Lambda t) w(t).
\end{align}
Now consider the Duhamel formula for $\widehat{w}$ in (\ref{PK_system}) in terms of $\widehat{f}_w(t, \xi)$:
\begin{equation}\label{eq:Duhamel-w}
\partial_t \widehat{f}_w(t, \xi)= \int m(\xi, \eta) \exp{(i\phi_w(\xi, \eta)t)} \widehat{f}_w(\xi-\eta, t) \widehat{f}_w(\eta, t) \, d\eta,
\end{equation}
where the phase 
$$\phi_w(\xi, \eta)=|\xi|-|\xi-\eta|-|\eta|.$$

We need the following two definitions.
\begin{Def}(from \cite{CM} and [Definition 2.1, \cite{tao}])\label{def:CM}
A symbol $m(\xi, \eta)$ belongs to the Coifman-Meyer class if
\begin{align*}
|\partial^\alpha_\xi 	\partial_\eta^\beta m(\xi, \eta)| \le \frac{C}{(|\xi|+|\eta|)^{(|\alpha|+|\beta|)}}
\end{align*}
{for a finite number of multi-indeces $(\alpha, \beta)$ depending on the dimension $d$ (here $d=3$) of the physical space.}
\end{Def}

\begin{Def}[from \cite{Pusateri}, Definition 2.1]\label{def:symbols}
A symbol $m(\xi, \eta)$ belongs to the class $\mathcal{B}_s$ if:
\begin{itemize}
\item it is homogeneous of degree $s$; 
\item it is smooth outside $\{\xi=0\} \cup \{\xi-\eta=0\} \cup \{\eta=0\}$;
\item {for any labeling $(\xi_1, \xi_2, \xi_3)$ of the Fourier variables $(\xi, \xi-\eta, \eta)$, there exists a smooth function $\mathcal{A}\left(x,y,z\right)$ of the variables $x,y,z$ such that} $$\text{for  } |\xi_1| \ll |\xi_2|, |\xi_3| \approx 1, \quad m(\xi_1, \xi_2, \xi_3)=\mathcal{A}\left(|\xi_1|, \dfrac{\xi_1}{|\xi_1|}, \xi_2\right).$$
\end{itemize}
\end{Def}

{
\begin{Rmk}\label{rmk-holder}
Bilinear pseudoproduct operators $T_m$ associated to symbols $m(\xi, \eta)$ belonging to the Coifman-Meyer class and symbols in $\mathcal{B}_0$ are both bounded in $L^r$, i.e.
$$T_m: L^p \times L^q \rightarrow L^r, \quad \frac{1}{r}=\frac{1}{p}+\frac{1}{q}, \quad 1 < p,q < \infty, \; 0 < r < \infty.$$
{
More precisely, if $m \in \mathcal{B}_0$, then
\begin{align*}
\|T_m (f,g) \|_{L^r} \lesssim \|f\|_{L^p} \|g\|_{L^q}.
\end{align*}
Moreover, if $m \in \mathcal{B}_s$ for $s \ge 0$ and $k$ is an integer, then
\begin{align*}
\|\Lambda^k T_m (f,g) \|_{L^r} \lesssim \|f\|_{W^{s+k,p}} \|g\|_{L^q}+ \|f\|_{L^{p}} \|g\|_{W^{s+k, q}}.
\end{align*}
}
In the case of Coifman-Meyer operators, the same holds for $1<p, q \le \infty$.
More precise results will be discussed in the Appendix, see also \cite{CM, GMS1, Pusateri}.
\end{Rmk}
\begin{Rmk}[On the classes of symbols]\label{rmk:symbols}
As widely discussed and showed in \cite{Pusateri}, symbols in $\mathcal{B}_0$ in Definition \ref{def:symbols} are Coifman-Meyer (Definition \ref{def:CM}) outside the axes $\xi=0, \, \eta=0, \, \eta-\xi=0$. Throughout the paper, we sometimes work with smooth (Coifman-Meyer) operators which could not be homogeneous, while in other cases we deal with homogeneous symbols which have (weak) singularities as the $\mathcal{B}_0$ class.
We finally remark that, as pointed out in \cite{GMS1}, if $m(\xi, \eta)$ is a Coifman-Meyer multiplier, so is $m_\lambda(\xi, \eta):=m(\lambda \xi, \lambda \eta)$, so that the estimate of Remark \ref{rmk-holder} is independent of $\lambda$.
\end{Rmk}}

\begin{Def}[Resonant sets]\label{def:resonant-sets}
The time and space resonant sets are defined respectively as
$$\mathcal{T}=\{(\xi, \eta): \phi_w(\xi, \eta)=0\}, \quad \mathcal{S}=\{(\xi, \eta):\nabla_\eta \phi_w(\xi, \eta)=0\}.$$ The time-space resonant set is then
$$\mathcal{R}=\mathcal{T}\cap \mathcal{S}.$$
\end{Def}
{Intuitively, a bilinear form $m(\xi, \eta)$ is \emph{nonresonant} with respect to a linear system (whose oscillations are represented by the phase $\phi(\xi, \eta)$) if the time-space resonant set is empty and the time-space resonant set of the additional bilinear forms $\widetilde m$ that will be created by the integration by parts steps (of the space-time resonant method) is not ``too large''.}

\begin{Def}[Nonresonant Bilinear Form]\label{def:Nonresonant-Bil-Form}
Equation (\ref{eq:Duhamel-w}) is called \emph{nonresonant} if 
$$m(\xi, \eta)=a(\xi, \eta)\phi_w(\xi, \eta)+b(\xi, \eta) \nabla_\eta \phi_w(\xi, \eta),$$
where $a(\xi, \eta) \in \mathcal{B}_1$ and $b(\xi, \eta) \in \mathcal{B}_0$.
Two more purely technical requirements on $a(\xi, \eta), b(\xi, \eta)$ are presented in [\cite{Pusateri}, Definition 2.2].
\end{Def}

The main point of Definition \ref{def:Nonresonant-Bil-Form} is that the symbol $m(\xi, \eta)$, then the principal part of the pseudoproduct operator $T_m$ vanishes on the space-time resonant set $\mathcal{R}$. 

{\begin{Rmk}[Nonresonant Bilinear Forms contain Classical Null Forms]
We point out that setting $a(\xi, \eta)=0$ in the above formula, one recovers the classical null condition (see for instance \cite{Klainerman1} for the characterization classical null forms). We would like to provide explanations on this connection at least in some specific case. Let us consider the following semilinear equation for acoustic waves in the unknown $u(t, x): \mathbb{R}^+ \times \mathbb{R}^3 \rightarrow \mathbb{R}$
\begin{align*}
\de_{tt} u - \Delta u = Q_{\iota}(u, u),  \quad \text{with} \quad \iota \in \{0, 0i, ij\}, \quad i,j \in \{1,2,3\},
\end{align*}
where, for any $u, v$, 
\begin{align*}
Q_0(u, v)= \de_t u \de_t v - \nabla  u \cdot \nabla v, \quad Q_{0i}(u, v)= \de_t u \de_i v - \de_i  u \de_t v, \quad Q_{ij}(u, v)= \de_i u \de_j v - \de_j  u \de_i  v.
\end{align*} 
As underlined in \cite{Pusateri}, page 1501, any of the above quadratic forms is associated with a pseudoproduct operator $T_m$ whose symbol $m=m(\xi, \eta)$ is \rm{nonresonant} in the sense of Definition \ref{def:Nonresonant-Bil-Form}. We explain this claim below in the case of $Q_0(u, v)$. We consider then the following semilinear wave equation
\begin{align}\label{eq:eqjohn}
\de_{tt} u - \Delta u =Q_0(u,u)= (\de_t u)^2 - |\nabla u|^2. 
\end{align}
As already mentioned in the introduction, this is a typical example of semilinear wave equation with quadratic source satisfying the classical null condition and it is due to F. John \cite{John}. We want to show that the Fourier formulation of equation \eqref{eq:eqjohn} naturally rewrites in terms of the pseudoproduct $T_m$ where the bilinear symbol is $m(\xi, \eta) = 1 \pm \frac{(\xi-\eta)\cdot \eta}{|\xi-\eta| |\eta|}$. Afterwards it can be easily checked that such a symbol $m(\xi, \eta)$ fulfills the condition of Definition \ref{def:Nonresonant-Bil-Form} and $T_m$ is therefore a nonresonant bilinear form. \\
We introduce the following change of variable 
\begin{align*}
w= \de_t u - i \Lambda u, \quad \Lambda=|\nabla|=Op(|\xi|). 
\end{align*}
Note that the solution $u(t, x)$ to equation \eqref{eq:eqjohn} is real valued, so that we can write $\overline{w}=\de_t u + i \Lambda u$ (where $\overline{w}$ denotes the complex conjugate). In the new variables $w, \overline{w}$, equation \eqref{eq:eqjohn} reads
\begin{align*}
\de_t w + i \Lambda w = \left(\frac{w+\overline{w}}{2}\right)\cdot \left(\frac{w+\overline{w}}{2}\right) + \frac{\nabla}{|\nabla|} \left(\frac{w-\overline{w}}{2}\right) \cdot \frac{\nabla}{|\nabla|} \left(\frac{w-\overline{w}}{2}\right). 
\end{align*}
Passing to the Fourier formulation, it is now easy to see that one has
\begin{align*}
\de_t \widehat{w}+ i |\xi| \widehat{w}&= \int \left(1 - \frac{(\xi-\eta)\cdot \eta}{|\xi-\eta| |\eta|}\right) \left(\widehat{w}(\xi-\eta) {\widehat{w}}(\eta) + \overline{\widehat{w}}(\xi-\eta) \overline{\widehat{w}}(\eta)\right) \, d\eta\\
&\quad +  \left(1 + \frac{(\xi-\eta)\cdot \eta}{|\xi-\eta| |\eta|}\right) \left(\widehat{w}(\xi-\eta) {\overline{\widehat{w}}}(\eta) + \overline{\widehat{w}}(\xi-\eta) {\widehat{w}}(\eta)\right) \, d\eta. 
\end{align*}
Therefore we observe from this example that the bilinear pseudoproduct operator arises naturally in the context of first order hyperbolic systems with quadratic source satisfying the time and space nonresonant condition. In fact, it is simply a consequence of some change of variables and the Fourier formulation of the quadratic source: the Fourier analysis is crucial to take advantage from the oscillations away from $\mathcal{S}$ (or the non crossing away from $\mathcal{T}$) and, in order ``to integrate away from $\mathcal{S}$ (or $\mathcal{T}$)'', one needs to cut-off in Fourier space. Besides its complicated form, the pseudoproduct operator   is a natural tool arising from the Fourier representation and generalization of very concrete and classical null forms for the semilinear wave equation (see also \cite{bernicot}).
\end{Rmk}}

The aim of this section is to provide a first example of global in time well-posedness for a partially dissipative hyperbolic system (\ref{PK_system}), whose partial dissipation violates the [SK] coupling condition.

In the following, we prove the result stated below.
\begin{thm}{[Global existence and decay estimates for the first weaker [SK] model]}\label{thm:PSK}
Consider the Cauchy problem associated with system (\ref{PK_system}), where $m(\xi, \eta)$ is a Nonresonant Bilinear Form as in Definition \ref{def:Nonresonant-Bil-Form}. Let the initial data $U(1):=(u(1), v(1), w(1))$ belong to $H^N(\R^3)$. There exist $N>\frac{5}{2}$ large enough and $\varepsilon>0$ small enough such that, introducing 
{
\begin{equation}
\label{def:norm-initial-data}
E_N=\max \{\|U(1)\|_{L^1}, \, \|xU(1)\|_{H^2}+\|\Lambda x^2 U(1)\|_{H^1}+ \|U(1)\|_{H^N}\},
\end{equation}}
and assuming that
$$E_N \le \varepsilon,$$
system (\ref{PK_system}) admits a unique global in time solution 
$$U(t):=(u(t), v(t), w(t)) \in C([1, \infty); H^N(\mathbb{R}^3)) \cap C^1([1, \infty); H^{N-1}(\mathbb{R}^3)).$$
Moreover, defining 
\begin{equation}
\label{X_norm}
\begin{aligned}
M_0(t)&=\sup_{t \ge 1} \{ t^{\frac{3}{4}} \|u(t)\|_{H^N} + t^{\frac{3}{4}-\varepsilon} \|v(t)\|_{H^N} +  t^{\frac{3}{2}-2\varepsilon} |u(t)|_{L^\infty} + t^{\frac{3}{4}-\varepsilon} |v(t)|_{L^\infty}\\
&\quad +t^{-\varepsilon}\|w(t)\|_{H^N}+t |w(t)|_{L^\infty} + t |R w(t)|_{L^\infty} + t^{-\gamma} \|x f_w(t)\|_{L^2}\\ 
& \quad + \| \Lambda x f_w(t)\|_{H^1}+t^{-1} \| |x|^2 \Lambda f_w (t)\|_{H^1}\},
\end{aligned}
\end{equation}
then the following decay estimate holds true:
\begin{align*}
M_0(t) \le C E_N.
\end{align*}
\end{thm}

\begin{Rmk}[{[SK] condition and the Conservative-Dissipative Form}]\label{rmk-CD-form}
We recall that any partially dissipative hyperbolic system satisfying [SK] can be put in the so-called Conservative-Dissipative Form (C-D) (see \cite{BHN}) by means of a standard change of coordinates based on the projection onto the kernel of the linear source term and its complementing, see \cite{BHN}, \cite{Bianchini} for further details. In particular, the first two equations of system (\ref{linear-PK-system}) are in (C-D) Form. Its relevance will be clear in the spectral analysis of Section \ref{subsec: spectral-analysis}, where the variable $v$ is showed to decay faster, thus playing a key role for closing the bootstrap.
Therefore, although system \eqref{linear-PK-system} is simple, it captures the fundamental structure of a quite general partially dissipative system satisfying [SK]. The only actual simplification is the replacement of the symbol of the vectorial operator $\nabla=(\partial_{x_1}, \partial_{x_2}, \partial_{x_3})$ with $|\xi|=\sqrt{\xi_1^2+\xi_2^2+\xi_3^2}$, and the linear dissipative equations in \eqref{linear-PK-system} are a scalar reduction of the following more general dissipative system\begin{align*}
\begin{cases}
\partial_t u + \nabla \cdot v=0, \\
\partial_t v_1 + \partial_{x_1}u + v_1=0, \\
\partial_t v_2 + \partial_{x_2} u + v_2=0, \\
\partial_t v_3 + \partial_{x_3} u + v_3=0, \quad \text{where    }  v=(v_1, v_2, v_3).
\end{cases}
\end{align*}
This simplifies the low frequency analysis, but the qualitative properties of the system persist, so that in the end our linearized system \eqref{linear-PK-system} is quite representative. 
\end{Rmk}

\begin{Rmk}[A comment on the quadratic source of system (\ref{PK_system})]\label{rmk:nonlinearity}

The coupling between the dissipative variables $u, v$ and the transported variable $w$ in \eqref{PK_system} is due to the quadratic source term $uw$ in the equation for $v$.
The following proof also works for similar cases: for instance, one could couple the equation through the quadratic term $vw$ in the equation for slower variable $u$, so that the system reads
\begin{equation*}
\begin{cases}
\partial_t \widehat{u}+i|\xi| \widehat{v}=a_u \widehat{u^2} + b_u \widehat{v^2} + c_u \widehat{uv}  + d_u \widehat{vw},  \\
\partial_t \widehat{v}+i|\xi| \widehat{u}+\widehat{v}=a_v \widehat{u^2} + b_v \widehat{v^2} + c_v \widehat{uv},\\
\partial_t \widehat{w}+i|\xi| \widehat{w}= \widehat{T}_{m(\xi, \eta)}(w,w).
\end{cases}
\end{equation*}
Thus there is some flexibility on the choice of the coupling in the quadratic source, but there are also some limitations. Precisely, general quadratic nonlinearities involving the three variables $u, v, w$  cannot be treated with the simple bootstrap argument of this section, since the time decay is in general not enough to close the estimates and new ideas are needed. This is why a ``nonlinear Fourier analysis of the profile'' (the so-called \emph{space-time resonance method}, see \cite{GM, GMS2, Germain}) will be implemented in the next section on a more relevant partially dissipative hyperbolic system violating [SK].
 \end{Rmk}

The first step toward the proof of Theorem \ref{thm:PSK} is the analysis of the linear part of the system.

\subsection{Spectral analysis of the linearized system}\label{subsec: spectral-analysis}
Let us set $U=(u, v, w)$. The linear system (\ref{linear-PK-system}) in compact form reads
$$\partial_t \widehat{U}-E(i\xi)\widehat{U}=0,$$
where 
\begin{equation}
\label{Green_Partial_Kawashima}
\begin{aligned}
E(i\xi)=\left(\begin{array}{ccc}
0 & -i|\xi| & 0 \\
-i|\xi| & -1 & 0 \\
0 & 0 & -i|\xi|
\end{array}\right)=-i|\xi|A+B, \quad \text{with} \quad
A=\begin{pmatrix}
0 & 1 & 0 \\
1 & 0 & 0 \\
0 & 0 & 1
\end{pmatrix}, \;
B=\begin{pmatrix}
0 & 0 & 0 \\
0 & -1 & 0 \\
0 & 0 & 0 \\
\end{pmatrix}.
\end{aligned}
\end{equation}

We write here the eigenvalues associated to $E(i\xi)$, 

$$\lambda_1(\xi)=-\frac{1}{2}+\frac{\sqrt{1-4|\xi|^2}}{2}, \quad \lambda_2(\xi)=-\frac{1}{2}-\frac{\sqrt{1-4|\xi|^2}}{2}, \quad \lambda_3=-i|\xi|,$$

while the related eigenvectors are

$$
V_1(\xi)=\left(\begin{array}{c}
1\\
-i\Bigg(\frac{1-\sqrt{1-4|\xi|^2}}{2 |\xi|}\Bigg)\\
0
\end{array}\right),
\quad
V_2(\xi)=\left(\begin{array}{c}
1\\
-i\Bigg(\frac{1+\sqrt{1-4|\xi|^2}}{2 |\xi|}\Bigg)\\
0
\end{array}\right), \quad V_3=\left(\begin{array}{c}
0\\
0\\
1
\end{array}\right).$$

We analyse the low frequency regime. Then, fix $0<a\ll 1$ and consider $|\xi| \le a$, for $a$ small enough. The eigenvalues expansion yields the following expressions:

\begin{equation*}
\lambda_1^a(\xi)=-|\xi|^2+O(|\xi|^3), \quad \lambda_2^a(\xi)=-1+|\xi|^2+O(|\xi|^3).
\end{equation*}

Thus the first two equations in $u, v$ of system (\ref{linear-PK-system}) satisfy [SK]. On the other hand,
since the eigenvector $V_3$ belongs to the kernel of the linear source term $B$, then the linear part (\ref{linear-PK-system}) of system (\ref{PK_system}) is not [SK].
 
In the low frequency regime, the eigenprojectors associated to the eigenvalues near 0, i.e. $\lambda_1^a(\xi), \lambda_3(\xi)$, are given by:

\begin{equation}
\label{Total_P_zero}
P_1(i\xi)=\left(\begin{array}{ccc}
1+O(|\xi|^2) & -i|\xi|+O(|\xi|^2) & 0 \\
-i|\xi|+O(|\xi|^2)  & -|\xi|^2+O(|\xi|^3) & 0 \\
0 & 0 & 0
\end{array}\right), \quad
P_3=\left(\begin{array}{ccc}
0 & 0 & 0 \\
0 & 0 & 0 \\
0 & 0 & 1
\end{array}\right).
\end{equation}

The complementing projector reads:

\begin{equation}
\label{Total_P_negative}
P_2(i\xi)=\left(\begin{array}{ccc}
O(|\xi|^2) & i|\xi|+O(|\xi|^2) & 0 \\
i|\xi|+O(|\xi|^2)  & 1+|\xi|^2+O(|\xi|^3) & 0 \\
0 & 0 & 0
\end{array}\right).
\end{equation}

Now, let $\Gamma(t,x)$ be the Green function associated to system (\ref{linear-PK-system}), i.e. 
$$\Gamma(t,x)=\mathcal{F}^{-1}(\widehat{\Gamma}(t, \xi))=\mathcal{F}^{-1}(\exp(E(i\xi)t)).$$

From the previous expansions, we get an expression of the Green function in Fourier space:
\begin{align*}
\exp(E(i\xi)t)&=P_1(i\xi) \exp(-c|\xi|^2t + O(|\xi|^3)t) + P_2(i\xi) \exp(-ct+O(|\xi|^3))+P_3 \exp(-i|\xi|t),
\end{align*}
{for any universal constant $c>0$.}

\subsection{Decay estimates on the linearized system}\label{subsec:decay-linear}
In the low frequency regime, the spectral analysis of Section \ref{subsec: spectral-analysis} provides the following decomposition:
\begin{equation}
\label{PK_Green_function_decomposition}
\begin{aligned}
\widehat{\Gamma}(t)&=\widehat{K}(t)+\widehat{\mathcal{K}}(t)+\widehat{W}(t)\\&
\approx \exp(-c|\xi|^2 t)  \left(\begin{array}{ccc}
1 & |\xi| & 0 \\
|\xi| & |\xi|^2 & 0\\
0 & 0 & 0
\end{array}\right) 
+ \exp(-ct)  \left(\begin{array}{ccc}
|\xi|^2 & |\xi| & 0\\
|\xi| & |\xi|^2 & 0\\
0 & 0 & 0
\end{array}\right)
+ \exp(-i|\xi|t) \left(\begin{array}{ccc}
0 & 0 & 0\\
0 & 0 & 0\\
0 & 0 & 1
\end{array}\right).
\end{aligned}
\end{equation}

Therefore the solution 
\begin{align}\label{def:solU}
U(t)=(u(t), v(t), w(t))=\Gamma(t)U_0 
\end{align}
to the linear system (\ref{linear-PK-system}) with initial data $U(t=1)=U_0$ can be decomposed as follows:
\begin{equation}
\label{dec:U}
U(t)=\Gamma(t)U_0=K(t)U_0+\mathcal{K}(t)U_0+W(t)U_0.
\end{equation}

Define the left eigenprojectors onto the null space of the linear source term $L_1=(1, 0, 0), \; L_3=(0, 0, 1),$ and its complementing $L_{2}=(0, 1, 0).$
We will use the notation $\Lambda:=Op(|\xi|)$ (but also for $|\xi-\eta|$ and $|\eta|$), while $R$ denotes the Riesz transform. Occasionally we identity symbols with related operators.\\
Let us recall the decay estimates from \cite{BHN}.

\begin{Prop}\label{prop:decay-SK}
Let $U(t)$ be defined in (\ref{dec:U}), with $U_0 \in L^1(\R^3) \cap H^s(\R^3).$ Then, for any multi-index $\beta$ with $|\beta| \le s$, there exists $C>0$ such that the following the following decay estimates hold for $t \ge 1$:
\begin{equation}
\label{PK_Decay_estimates}
\begin{aligned}
& i)\;  \|L_1 D^\beta K(t) U_0\|_{L^2} \le C \min \{1, t^{-\frac{3}{4}-\frac{|\beta|}{2}}\} \|L_1 U_0\|_{L^1} + C \min \{1, t^{-\frac{5}{4}-\frac{|\beta|}{2}} \} \|L_2 U_0\|_{L^1}, \\
& ii)\; \|L_2 D^\beta K(t) U_0\|_{L^2} \le C \min \{1, t^{-\frac{5}{4}-\frac{|\beta|}{2}}\} \|L_1 U_0\|_{L^1} + C \min \{1, t^{-\frac{7}{4}-\frac{|\beta|}{2}} \} \|L_2 U_0\|_{L^1}, \\
& iii) \; |L_1 D^\beta K(t) U_0|_{L^\infty} \le C \min \{1, t^{-\frac{3}{2}-\frac{|\beta|}{2}}\} \|L_1 U_0\|_{L^1} + C \min \{1, t^{-2-\frac{|\beta|}{2}} \} \|L_2 U_0\|_{L^1}, \\
& iv) \;|L_2 D^\beta K(t) U_0|_{L^\infty} \le C \min \{1, t^{-2-\frac{|\beta|}{2}}\} \|L_1 U_0\|_{L^1} + C \min \{1, t^{-\frac{5}{2}-\frac{|\beta|}{2}} \} \|L_2 U_0\|_{L^1}, \\
& v)\; \|D^\beta \mathcal{K}(t) U_0\|_{L^2} \le C \exp(-ct) \|D^\beta U_0\|_{L^2}.\\
\end{aligned}
\end{equation}
\end{Prop}

Some a priori bounds on $w$ are listed in \cite{Pusateri}.
\begin{Prop}\label{prop:PS} 
Let $w(t)$ be the third component of the solution $U(t)$ to system \eqref{linear-PK-system}, as in \eqref{def:solU}, with initial data $w_0=w(t=1)$ such that there exists $C>0$ such that $\|x w_0\|_{H^2} + \|\Lambda x^2 w_0\|_{H^1} + \|w_0\|_{H^N} \le C$. Let $f_w(t)$ {be} the \emph{profile} of $w(t)$, which is defined as follows:
\begin{equation}
\label{w_profile}
f_w(t):=e^{i\Lambda t} w(t).
\end{equation}
Then, for a suitable large $N$ and arbitrarily small fixed constants $\varepsilon, \gamma$, the following decay rates hold for $t \ge 1$:
\begin{equation}
\label{PK_Decay_estimates_1}
\begin{aligned}
& i) \;\|L_3 W(t)U_0\|_{H^N}= \|w\|_{H^N} \lesssim t^\varepsilon, \\
& ii) \; |L_3 W(t)U_0|_{L^\infty}= |w|_{L^\infty} \lesssim t^{-1}, \quad |L_3 R W(t)U_0|_{L^\infty}= |R w|_{L^\infty} \lesssim t^{-1},\\
& iii) \; \|xf_w\|_{L^2} \lesssim t^\gamma, \\
& iv) \; \|\Lambda x f_w\|_{H^1} \lesssim 1, \\
& v) \; \||x|^2 \Lambda f_w\|_{H^1} \lesssim t.
\end{aligned}
\end{equation}
\end{Prop}

\subsection{Decay estimates on the nonlinear system}\label{subsect:nonlinear1}

The equation for $w$ in (\ref{PK_system}) is independent of the dissipative variables $u, v$ and it has already been studied in \cite{Pusateri}. Therefore, we will simply implement the estimates obtained in \cite{Pusateri} in the dissipative part of system \eqref{PK_system}. More precisely, we use the following result from \cite{Pusateri}.

\begin{Prop}\label{prop:PS-nonlinear}
Let $w$ be the solution to 
$$\partial_t \widehat{w}+i|\xi| \widehat{w}=\widehat{T}_{m(\xi, \eta)} (w,w), \quad \widehat{w}(1, \xi)=\widehat{w}_0(\xi),$$
where $m(\xi, \eta)$ is a Nonresonant Bilinear Form as in Definition \ref{def:Nonresonant-Bil-Form}.
Consider the equation for the profile $\widehat{f}_w(t, \xi)$ in (\ref{w_profile}), i.e.
$$\widehat{f}_w(t, \xi)=\widehat{w}_0(\xi)+\int_1^t \widehat{T}_{m(\xi, \eta)}(w,w) \, ds.$$
There exists $\varepsilon_0 >0$ small enough such that, for 
$$\|x w_0\|_{H^2} + \|\Lambda x^2 w_0\|_{H^1} + \|w_0\|_{H^N} \le \varepsilon_0,$$
then the following decay estimates hold for $t \ge 1$:
\begin{equation*}
\begin{aligned}
&\|w(t)\|_{H^N} \lesssim t^\varepsilon{(\eps_0 + M_0^2(t))}, \quad 
|w(t)|_{L^\infty} \lesssim t^{-1}{(\eps_0 + M_0^2(t))}, \quad |R w(t)|_{L^\infty} \lesssim t^{-1}{(\eps_0 + M_0^2(t))},\\
&\|xf_w(t)\|_{L^2} \lesssim t^\gamma {(\eps_0 + M_0^2(t))}, \quad
\|\Lambda x f_w(t)\|_{H^1} \lesssim {(\eps_0 + M_0^2(t))}, \quad
\||x|^2 \Lambda f_w(t)\|_{H^1} \lesssim t {(\eps_0 + M_0^2(t))}, \\
&\|\partial_t f_w\|_{L^2} \lesssim t^{-2+\delta} {(\eps_0 + M_0^2(t))}, \quad \|x\partial_t f_w\|_{L^2} \lesssim t^{-1+\delta} {(\eps_0 + M_0^2(t))},
\end{aligned}
\end{equation*}
for a suitable large $N$ and arbitrarily small fixed constants $\varepsilon, \delta, \gamma$.
\end{Prop}

\begin{proof}[Proof of Theorem \ref{thm:PSK}]
\textbf{Decay estimates in $H^N$.}
We apply the Duhamel formula to system (\ref{PK_system}), recalling that for the slower variable $u$ we expect a decay $\|u(t)\|_{H^N} \approx t^{-\frac{3}{4}}$ as in the linear case in (\ref{PK_Decay_estimates}), while the expected rate for $v$ is quite worse $\|v(t)\|_{H^N} \approx t^{-\frac{3}{4}+\varepsilon}$ with respect to the linear case, because of the mixed term $d_v uw$. Then,
\begin{align*}
\|u(t)\|_{H^N} & \le CE_N \min \{1, t^{-\frac{3}{4}}\} + C M_0^2(t) \int_1^t \exp(-c(t-s)) \min \{1, s^{-\frac{3}{4}}\} \, ds \\
&\quad  + \int_1^t \min\{1, (t-s)^{-\frac{3}{4}}\} \|u(s)\|^2_{H^N} \, ds + \int_1^t \min\{1, (t-s)^{-\frac{3}{4}}\} \, \|u(s)\|_{H^N} \|v(s)\|_{H^N} \, ds\\
&\quad + \int_1^t \min\{1, (t-s)^{-\frac{3}{4}}\} \, \|v(s)\|^2_{H^N}\, ds \\
& \le C \min \{1, t^{-\frac{3}{4}}\} + C M_0^2(t) \int_0^t \exp(-c(t-s)) \min \{1, s^{-\frac{3}{4}}\} \, ds \\
& \quad + C M_0^2(t) \int_1^t \min\{1, (t-s)^{-\frac{3}{4}}\} \min \{1, s^{-\frac{3}{2}}\} \, ds + CM_0^2(t) \int_1^t \min\{1, (t-s)^{-\frac{3}{4}}\} \, \min \{1, s^{-\frac{3}{2}+\varepsilon} \} \, ds\\
&\quad + CM_0^2(t) \int_1^t \min\{1, (t-s)^{-\frac{3}{4}}\} \, \min \{1, s^{-\frac{3}{2}+2\varepsilon} \} \, ds. 
\end{align*}

Thanks to [\cite{BHN}, Lemma 5.2], one gets:
\begin{equation}
\label{PK_estimate_u}
\|u(t)\|_{H^N} \le CE_N \min \{1, t^{-\frac{3}{4}}\} +  Ct^{-\frac{3}{4}} M_0^2(t).
\end{equation}
{
The estimate for the dissipative variable is analogous. For some universal constant $C>0$, we have that
\begin{align*}
\|v(t)\|_{H^N} & \le C \min \{1, t^{-\frac{5}{4}}\}E_N + C M_0(t) E_{N} \int_1^t \exp(-c(t-s)) \min \{1, s^{-\frac{3}{4}+\varepsilon}\} \, ds \\
&\quad + C M_0(t)^2 \int_1^t \min\{1, (t-s)^{-\frac{5}{4}}\} \min \{1, s^{-\frac{3}{2}}\} \, ds + CM_0(t)^2 \int_1^t \min\{1, (t-s)^{-\frac{5}{4}}\} \, \min \{1, s^{-\frac{3}{2}+\varepsilon} \} \, ds\\
&\quad + CM_0(t)^2 \int_1^t \min\{1, (t-s)^{-\frac{5}{4}}\} \, \min \{1, s^{-\frac{3}{2}+2\varepsilon}  \} \, ds + C\int_1^t \min\{1, (t-s)^{-\frac{5}{4}}\} \|u w(s)\|_{L^1} \, ds.
\end{align*}
The first four integrals are $O( t^{-\frac{5}{4}})$. Let us consider the last one. We have
\begin{align*}
 \int_1^t  \min\{1, (t-s)^{-\frac{5}{4}}\} \|uw(s)\|_{L^1} \, ds & \le C \int_1^t \min\{1, (t-s)^{-\frac{5}{4}}\} \|u(s)\|_{L^2} \|w(s)\|_{L^2} \, ds \\
& \le C \int_1^t \min\{1, (t-s)^{-\frac{5}{4}}\} \|u(s)\|_{H^N} \|w(s)\|_{H^N} \, ds\\
&  \le C M_0(t)^2 \int_1^t \min\{1, (t-s)^{-\frac{5}{4}}\} \min\{1, s^{-\frac{3}{4}+\varepsilon}\} \, ds\\
& \le C t^{-\frac{3}{4}+\varepsilon}M_0(t)^2.
\end{align*}
Note in particular that this last term is responsible for the $\eps$ loss in the time-decay rate of the dissipative variable $v(t)$. In fact, we know from Proposition \ref{prop:PS-nonlinear} (and from \cite{Pusateri}) that in $H^N$ the wave variable $w(t)$ has a time growth of order $t^\varepsilon$. 
Putting all the terms together, we have
\begin{equation}
\label{PK_estimate_v}
\|v(t)\|_{H^N} \le C \min \{1, t^{-\frac{5}{4}}\}E_N + t^{-\frac{3}{4}+\varepsilon} M_0^2(t).
\end{equation}}

\textbf{Decay estimates in $L^\infty$.}
We start from the slower variable $u$:
\begin{align*}
|u(t)|_{L^\infty} & \le C \min \{1, t^{-\frac{3}{2}}\} + C M_0^2(t) \int_1^t \exp(-c(t-s)) \min \{1, s^{-\frac{3}{2}+2 \varepsilon}\} \, ds \\
&\quad + \int_1^t \min\{1, (t-s)^{-\frac{3}{2}}\} \|u(s)\|^2_{H^N} \, ds\ + \int_1^t \min\{1, (t-s)^{-\frac{3}{2}}\} \, \|u(s)\|_{H^N} \|v(s)\|_{H^N} \, ds\\
&\quad + \int_1^t \min\{1, (t-s)^{-\frac{3}{2}}\} \, \|v(s)\|^2_{H^N}\, ds \\
& \le C \min \{1, t^{-\frac{3}{2}}\} + C M_0^2(t) \int_1^t \exp(-c(t-s)) \min \{1, s^{-\frac{3}{2}+2 \varepsilon}\} \, ds \\
&\quad + C M_0^2(t) \int_1^t \min\{1, (t-s)^{-\frac{3}{2}}\} \min \{1, s^{-\frac{3}{2}}\} \, ds+ CM_0^2(t) \int_1^t \min\{1, (t-s)^{-\frac{3}{2}}\} \, \min \{1, s^{-\frac{3}{2}+\varepsilon} \} \, ds\\
&\quad + CM_0^2(t) \int_1^t \min\{1, (t-s)^{-\frac{3}{2}}\} \, \min \{1, s^{-\frac{3}{2}+2\varepsilon} \} \, ds \\
& \le C \min \{1, t^{-\frac{3}{2}}\} + Ct^{-\frac{3}{2}+2 \varepsilon} M_0^2(t).
\end{align*}

This implies that
\begin{equation}
\label{u_L_infty_estimate}
|u(t)|_{L^\infty} \le CE_N \min \{1, t^{-\frac{3}{2}}\} + C t^{-\frac{3}{2}+2 \varepsilon}M_0^2(t).
\end{equation}

Similarly, the dissipative variable decays as follows:


\begin{equation}
\label{v_L_infty_estimate}
|v(t)|_{L^\infty} \le CE_N \min \{1, t^{-2}\} + C t^{-\frac{3}{4}+\varepsilon}M_0^2(t).
\end{equation}
\end{proof}

\subsection{Shizuta-Kawashima with any quadratic source}\label{subsec:SK}
The result of this section can be applied to the following reduced version of system (\ref{PK_system}):
\begin{equation}
\label{K_system}
\begin{cases}
\partial_t \widehat{u}+i|\xi| \widehat{v}=a_u \widehat{u^2} + b_u \widehat{v^2} + c_u \widehat{uv}, \\
\partial_t \widehat{v}+i|\xi| \widehat{u}+\widehat{v}=a_v \widehat{u^2} + b_v \widehat{v^2} + c_v \widehat{uv},
\end{cases}
\end{equation}

where $(t, \xi)=(t, \xi) \in \mathbb{R}^+ \times \mathbb{R}^3$ and 
$a_u, b_u, c_u, a_v, b_v, c_v$ are constant values. System (\ref{K_system}) is partially dissipative and satisfies [SK] in the unknown variable $U=(u, v)$, as it can checked from the linear analysis, following exactly Section \ref{subsec: spectral-analysis}. This linearized [SK] model is already in the Conservative-Dissipative Form (see Remark \ref{rmk-CD-form}). In this section, we extend the result of \cite{BHN} in the three dimensional semilinear case, as an application of the method developed for (\ref{PK_system}). Notice indeed that in the general theory developed in \cite{BHN}, the equation for the slow variable $u$ is conservative. Here we are instead able to handle a general quadratic source in the equation for $u$.
Starting from the linearized system in Fourier variables, we have
$$\partial_t \widehat{U} - E(i\xi) \widehat{U} =0, \qquad
E(i\xi)=\left(\begin{array}{cc}
0 & -i|\xi| \\
-i|\xi| & -1
\end{array}\right).
$$

The method of Section \ref{subsec: spectral-analysis} applied to this system provides the following eigenvalues expansions in the low frequency regime $|\xi| \ll a$ for $a$ small enough:
\begin{equation}
\label{eig_expansion}
\lambda^a(\xi)=-|\xi|^2+O(|\xi|^3), \quad  \lambda_{-}^a(\xi)=-1+|\xi|^2+O(|\xi|^3), 
\end{equation}
with related eigenprojectors:
\begin{equation*}
P(i\xi)=\left(\begin{array}{cc}
1+O(|\xi|^2) & -i|\xi|+O(|\xi|^2)\\
-i|\xi|+O(|\xi|^2) & -|\xi|^2 + O(|\xi|^3)
\end{array}\right), \quad 
P_{-}(i\xi)=\left(\begin{array}{cc}
O(|\xi|^2) & i|\xi| + O(|\xi|^2)\\
 i|\xi| + O(|\xi|^2) & 1+|\xi|^2+O(|\xi|^3)
\end{array}\right).
\end{equation*}
The Green function in Fourier space is given by:
\begin{align*}
\widehat{\Gamma}(t, \xi)& =\exp(E(i\xi)t)\\
&=P(i\xi) \exp((-|\xi|^2+O(|\xi|^3)) t) + P_{-}(i\xi) \exp((-1+O(|\xi|^2)t)\\
&\le C \exp(-c|\xi|^2 t)  \left(\begin{array}{cc}
1 & |\xi| \\
|\xi| & |\xi|^2 
\end{array}\right) 
+ C \exp(-ct)  \left(\begin{array}{cc}
|\xi|^2 & |\xi| \\
|\xi| & |\xi|^2 
\end{array}\right) \\
&=: \widehat{K}(t)+\widehat{\mathcal{K}}(t).
\end{align*}
Then the solution 
$U(t)=(u(t), v(t))=\Gamma(t)U_0$ to the linear system (\ref{K_system}) with initial data $U(t=1)=U_0$ can be decomposed as follows:
$$U(t)=\Gamma(t)U_0=K(t)U_0+\mathcal{K}(t)U_0,$$
and the decay estimates  i) , ii), iii), iv), v) of Section \ref{subsec:decay-linear} hold.
Now define the norm of the initial data for system (\ref{K_system}):
\begin{equation*}
E_N=\max\{ \|u(0)\|_{L^1}+\|v(0)\|_{L^1}, \|u(0)\|_{H^N}+\|v(0)\|_{H^N} \},
\end{equation*}
and introduce the functional:
\begin{equation*}
M_0(t)=\sup_{0 \le \tau \le t} \{ \max\{1, \tau^{\frac{3}{4}}\} \|u(\tau)\|_{H^N} + \max\{1, \tau^{\frac{5}{4}}\} \|v(\tau)\|_{H^N} \}.
\end{equation*}
Applying the procedure of Section \ref{subsect:nonlinear1}, we can prove the following result.
\begin{thm}{[Global existence and decay estimates for the [SK] system]}\label{thm:SK}
Consider the Cauchy problem associated with system (\ref{K_system}) and initial data belonging to $H^N(\mathbb{R}^3),$ with $N>\frac{5}{2}$ and $E_N \le \varepsilon$ for some fixed $\varepsilon$ small enough. Then there exists a unique global in time solution 
$$U=(u, v) \in C(([0, \infty); H^N(\mathbb{R}^3)) \cap C^1([0, \infty); H^{N-1}(\mathbb{R}^3)),$$
with time decay estimates:
$$\|u(t)\|_s \lesssim t^{-\frac{3}{4}} E_N, \qquad \|v(t)\|_s \lesssim t^{-\frac{5}{4}} E_N.$$
\end{thm}

\section{A step towards models of physical interest: \\when dissipation prevents time resonances}\label{sec:second-model}
In this section, we consider a hyperbolic system with quadratic source, whose linear part is exactly the linear operator (\ref{linear-PK-system}), while the coupling between the dissipative variables $u, v$ and $w$ lies in the quadratic source term of the equation for $w$. The system is the following:
\begin{equation}
\label{PK_system_w}
\begin{cases}
\partial_t \widehat{u}+i|\xi| \widehat{v}=\widehat{v^2},  \\
\partial_t \widehat{v}+i|\xi| \widehat{u}+\widehat{v}= \widehat{v^2},\\
\partial_t \widehat{w}+i|\xi| \widehat{w}= \widehat{vw}+\widehat{T}_{m(\xi, \eta)}(w,w).
\end{cases}
\end{equation}
The source term of the equation for $w$ is the sum of a Nonresonant Bilinear Form for the wave operator (as in Definition \ref{def:Nonresonant-Bil-Form},
where $T_{m(\xi, \eta)}$ is the bilinear pseudoproduct in (\ref{def-pseudoproduct}))
and a coupling term with the dissipative variables.
We will show that system (\ref{PK_system_w}) is well behaved for long times, despite the facts that:\\
$\bullet$ the last equation (the wave propagator) does not dissipate at all and exhibits  a quadratic nonlinearity;\\
$\bullet$ {the two addends in the source of $w$, namely $vw$ and $T_m$, have different natures: $m$ is a Nonresonant Bilinear Form for the wave propagator, see Definition \ref{def:Nonresonant-Bil-Form}; $vw$ does not have any special structure, but $v$ is a \emph{dissipative variable}. Their behaviors are different, as they appear with different phases in the Duhamel formula and we need to care of them separately.}\\
In the rest of this section we prove the global in time existence and decay of the smooth solutions to system (\ref{PK_system_w}). In spite of the lack of dissipation of the linear part (system (\ref{PK_system_w}) does not satisfy [SK], see Section \ref{subsec: spectral-analysis}), the coupling term $vw$ is able to transfer enough dissipation to the equation for $w$ to prevent the formation of singularities. Although being a first case study, this result is far from being obvious since in general solutions to hyperbolic systems with quadratic source are known to blow up, \cite{John2}. We define the following norm:
\begin{equation}
\label{X_norm_w}
\begin{aligned}
M_0(t)&=\sup_{t \ge 1}  \{t^{\frac{3}{4}} \|u(t)\|_{H^N} + t^{\frac{5}{4}} \|v(t)\|_{H^N} +  t^{\frac{3}{2}} |u(t)|_{L^\infty} + t^{\frac{5}{2}} |v(t)|_{L^\infty}+ t^\frac 34 \|x v (t)\|_{H^N}\\
& \quad 
+t^{-\varepsilon}\|w(t)\|_{H^N}+t |w(t)|_{L^\infty}+
 t |R w(t)|_{L^\infty} + t^{-\gamma} \|x f_w(t)\|_{L^2}+ \| \Lambda x f_w(t)\|_{H^1}+t^{-1} \| |x|^2 \Lambda f_w (t)\|_{H^1} \}.
\end{aligned}
\end{equation}
The result of this section is stated below.
\begin{thm}{[Global existence and decay estimates for the 2nd model violating [SK]]}\label{thm:PSK2}
Consider the Cauchy problem associated with system (\ref{PK_system_w}) and initial data at initial time $t=1$ belonging to $H^N(\mathbb{R}^3).$There exists $\varepsilon >0$ small enough and $N>\frac{5}{2}$ large enough such that, for
$$E_N \le \varepsilon \quad \text{defined in} \quad \eqref{def:norm-initial-data},$$
system (\ref{PK_system_w}) admits a unique global in time solution
$$U=(u, v, w) \in C([1, \infty); H^N(\mathbb{R}^3)) \cap C^1([1, \infty); H^{N-1}(\mathbb{R}^3)),$$
with time decay estimates for $t \ge 1$ provided by the inequality
\begin{align*}
M_0(t) \le C E_N.
\end{align*}
\end{thm}

\begin{Rmk}[A comment on the choice of the model]
In this first investigative paper, we limit ourselves to the toy-model case of (\ref{PK_system_w}). This is a simple and \emph{ad hoc} system, whose nonlinearity contains a special term, a so-called Nonresonant Bilinear Form. Thanks to the spectral analysis of Section \ref{subsec: spectral-analysis}, we know that the variable $w$ is not dissipated at all, while dissipation involves the [SK]-variables $u, v$. Thus, the intuition is that when the nonlinearity is quadratic in $w$ (the non-dissipated variable), a nonresonant/null structure is needed, otherwise singularities occur. Notice however by Definition \ref{def:Nonresonant-Bil-Form} that the whole nonlinear term (the three sources of the three equations) is not nonresonant with respect to the full system \eqref{PK_system_w}, so the nonresonance condition alone (on the third equation) is not enough and dissipation is still important in this framework. This is why, besides being a simple and \emph{ad hoc} example, system \eqref{PK_system_w} is interesting for our scope, as a first concrete case where the structure of the nonlinearity and dissipation mutually cooperate with a successful outcome, as proved in Theorem \ref{thm:PSK2}. Since \eqref{PK_system_w} is neither [SK] nor completely nonresonant, the setting and the result are new. There are finally two last remarks.\\
$\bullet$ The equations for $u, v$ in \eqref{PK_system_w} show a very typical [SK] structure and are actually quite general.\\
$\bullet$ This method also applies to more general quadratic source terms depending on the three variables $u, v, w$, at the price of a consistent additional number of steps in terms of normal form transformations and weighted/localized estimates, because of the different (slower) decay rates of $u, v$. We avoided to detail this general case as it does not bring new interesting insights into the problem. Physically relevant models will be instead studied in a future work (see the introduction).
\end{Rmk}

\subsection{Decay estimates on the nonlinear system in the low frequency regime}
The spectral analysis and the decay estimates on the linearized system are provided in Section \ref{subsec:decay-linear}:
$$\|u(t)\|_{H^N} \approx t^{-\frac{3}{4}}, \quad \|v(t)\|_{H^N} \approx t^{-\frac{5}{4}}, \quad |u(t)|_{L^\infty} \approx t^{-\frac{3}{2}}, \quad |v(t)|_{L^\infty} \approx t^{-\frac{5}{2}}.$$

Now we need to estimate $\|xv\|_{H^{N}}$. 
\begin{lem}
Let $U(t)$ be defined in (\ref{dec:U}), with $U_0 \in L^1(\R^3) \cap H^s(\R^3)$ Then, for any multi-index $\beta$ with $|\beta| \le s$, there exists $C>0$ such that the following decay estimate holds for $t \ge 1$:
\begin{align*}
\|D^\beta x L_2 K(t) U_0\|_{L^2} &\le C \min\{1, t^{-\frac 34-\frac{|\beta|}{2}}\} \|L_1 U_0\|_{L^1} + C \min\{1, t^{-\frac 54-\frac{|\beta|}{2}}\} \|L_2 U_0\|_{L^1}\\
& {\quad + C \min\{1, t^{-\frac 54-\frac{|\beta|}{2}}\} \|x L_1 U_0\|_{L^1} + C \min\{1, t^{-\frac 74-\frac{|\beta|}{2}}\} \|xL_2 U_0\|_{L^1}}.
\end{align*}
Moreover, let $v(t)$ be the second component of the solution $U(t)=(u(t), v(t), w(t))$ to the first two equations of \eqref{PK_system_w}. It holds that 
$$\| x L_2 K(t) v(t)\|_{H^N} \lesssim t^{-\frac 34}{(E_N+M_0^2(t))}.$$
\end{lem}
\begin{proof}
We compute
\begin{align*}
\|D^\beta  xL_2 K(t) U_0\|_{L^2}^2 &  \lesssim \int_{\R^3} \left|\de_\xi (|\xi| \exp(-c|\xi|^2 t))\right|^2 |\widehat{u}_0|^2 |\xi|^{2\beta} \, d\xi + \int_{\R^3}  |\xi|^2 \exp(-c|\xi|^2 t) |\de_\xi \widehat{u}_0|^2 |\xi|^{2\beta} \, d\xi\\
& \quad +  \int_{\R^3} \left|\de_\xi (|\xi|^2 \exp(-c|\xi|^2 t))\right|^2 |\widehat{v}_0|^2 |\xi|^{2\beta} \, d\xi + \int_{\R^3}  |\xi|^4 \exp(-c|\xi|^2 t) |\de_\xi \widehat{v}_0|^2 |\xi|^{2\beta} \, d\xi\\
& \lesssim \int_{\R^3}  \exp(-c|\xi|^2 t) |\widehat{u}_0|^2  |\xi|^{2\beta} \, d\xi + t^2  \int_{\R^3}  |\xi|^4 \exp(-c|\xi|^2 t) |\widehat{u}_0|  |\xi|^{2\beta} \, d\xi \\
& \quad + \int_{\R^3}  |\xi|^2 \exp(-c|\xi|^2 t) |\de_\xi \widehat{u}_0|^2  |\xi|^{2\beta} \, d\xi \\
& \quad + \int_{\R^3}  |\xi|^2 \exp(-c|\xi|^2 t) |\widehat{v}_0|^2  |\xi|^{2\beta} \, d\xi + t^2  \int_{\R^3}  |\xi|^6 \exp(-c|\xi|^2 t) |\widehat{v}_0|  |\xi|^{2\beta} \, d\xi \\
& \quad + \int_{\R^3}  |\xi|^4 \exp(-c|\xi|^2 t) |\de_\xi \widehat{v}_0|^2  |\xi|^{2\beta} \, d\xi \\
&\lesssim t^{-\frac 3 2-|\beta|}\|u_0\|_{L^1}^2+ t^{-\frac 52-|\beta|} \|x u_0\|_{L^1}^2+t^{-\frac 52-|\beta|}\|v_0\|_{L^1}^2 + t^{-\frac 72-|\beta|} \|xv_0\|_{L^1}^2.
\end{align*}
For the nonlinear equation, we simply observe that
\begin{align*}
\|x  v\|_{H^N} & \lesssim \min\{1, t^{-\frac 3 4}\} E_N +  \int_1^t (t-s)^{-\frac 34} \|v\|_{H^N}^2 \, ds + \int_1^t (t-s)^{-\frac 54} \|xv\|_{H^N} \|v\|_{H^N} \, ds\\
& \lesssim  \min\{1, t^{-\frac 3 4}\} E_N + \min\{1, t^{-\frac 34}\} M_0^2(t).
\end{align*}
\end{proof}
Now we are left with the estimates for $w$. We treat each term of the norm (\ref{X_norm_w}) separately. 

In the rest of this section, we prove Theorem \ref{thm:PSK2}.
Let us start from the Duhamel formula for $w(t)$,
\begin{equation}
\label{def:Duhamel-formula-w}
w(t)=w_0+\int_1^t  \exp(i \Lambda (t-s)) T_{m(\xi, \eta)}(w, w) \, ds + \int_1^t  \exp(i \Lambda (t-s))  vw \, ds.
\end{equation}

We will use also the \emph{profile} for the three variables in Fourier space $\widehat{u}, \widehat{v}, \widehat{w}$:
\begin{equation}\label{def:profiles}
\begin{aligned}
\widehat{f}(\xi, t)=
\begin{pmatrix}
\widehat{f}_u \\ \widehat{f}_v \\ \widehat{f}_w 
\end{pmatrix}
=\exp(-E(i\xi)t) 
\begin{pmatrix}
\widehat{u}\\
\widehat{v}\\
\widehat{w}
\end{pmatrix},
\end{aligned}
\end{equation}
with $E(i\xi)$ in (\ref{Green_Partial_Kawashima}).
In the low frequency regime, we recall the following expansions from Section \ref{subsec:decay-linear}:
\begin{align*}
\widehat{U}&=\begin{pmatrix}
\widehat{u}\\
\widehat{v}\\
\widehat{w}
\end{pmatrix}= \exp(E(i\xi)t) 
\begin{pmatrix}
\widehat{f}_u \\ \widehat{f}_v \\ \widehat{f}_w 
\end{pmatrix}
=\begin{pmatrix}
\begin{pmatrix}
1 & -i|\xi| \\
-i|\xi|  & -|\xi|^2 
\end{pmatrix} \exp(-|\xi|^2t)  & 0 \\
0 & \exp(-i|\xi|t)
\end{pmatrix} \begin{pmatrix}
\widehat{f}_u \\ \widehat{f}_v \\ \widehat{f}_w 
\end{pmatrix}+O(|\xi|^2).
\end{align*}
Then at main order (the remainders decay faster in $L^2$ and $L^\infty$; in fact, the analysis of the nonlinear equations for $u,v$ is already known and detailed in Section \ref{subsec:SK}) one gets:
\be\label{profile-main-order}
\begin{aligned}
 \widehat{u}(\xi, t) = \exp(-|\xi|^2t) \widehat{f}_u, \quad  \widehat{v}(\xi, t)= |\xi| \exp(-|\xi|^2 t) \widehat{f}_u, \quad \widehat{w}(\xi, t)=\exp(-i|\xi| t) \widehat{f}_w.
\end{aligned}
\ee
Now we write the Duhamel formula for the profile $\widehat{f}_w$:
\be \label{form:Duhamel-w}\begin{aligned}
\widehat{f}_w(\xi, t)  = \widehat{w}_0(\xi) & + \int_1^t  \exp(i|\xi|s)\widehat{T}_{m(\xi, \eta)}(\exp(-i |\xi-\eta| s) \widehat{f}_w, \exp(-i |\eta| s)  \widehat{f}_w) \, ds \\
& + \int_1^t \int \exp(i\phi(\xi, \eta)s) \widehat{f}_w(\xi-\eta, s)|\eta| \widehat{f}_u(\eta, s) \, d\eta \, ds,
\end{aligned}\ee
where we used (\ref{profile-main-order}) and we recall from Section \ref{subsec: spectral-analysis} that 
\begin{equation}
\label{def:phase-low-frequencies}
\phi(\xi, \eta)=|\xi|-|\xi-\eta| +i \left( \frac 12 - \frac 12 \sqrt{1-4|\eta|^2}\right).  
\end{equation}
It is easy to see that the time-resonant set of the phase $\phi$ denoted as $\mathcal{T}_\phi=\{ (\xi, \eta) \,| \, \eta=0\}$ is quite big, while the space-resonant set is given by 
\begin{align}\label{def:space-set}
\mathcal{S}_\phi=\{ (\xi, \eta) \,| \, \de_\eta \phi = 0\}=\left \{ (\xi, \eta) \, \Big| \, \frac{\xi-\eta}{|\xi-\eta|} +  \frac{2i\eta}{\sqrt{1-4|\eta|^2}}=0\right\} =\emptyset.
\end{align}

This implies that the space-time resonant set is also empty $\mathcal{R}_\phi=\emptyset$, which allows us to close the estimates by relying on the space-time resonant method. Notice that $\mathcal{T}_\phi$ and $\mathcal{R}_\phi$ are the time and space resonant sets that are only related to the third equation of system  \eqref{PK_system_w}, namely $w$ and that is not true that the space-time resonant set of the whole system (with all the source terms) is empty. 

Notice also that the first two terms of (\ref{def:Duhamel-formula-w}) and (\ref{form:Duhamel-w})  have already been estimated in every functional space involved in $M_0(t)$ in (\ref{X_norm_w}). These results, due to \cite{Pusateri}, are listed in Proposition \ref{prop:PS-nonlinear}. Therefore, we will only focus on the last integral, 
\begin{equation}
\label{int-last-w}
\int_1^t  \exp(i \Lambda (t-s))  vw \, ds.
\end{equation}

Similarly, to estimate the profile $f_w$, we will only write down the computations for
\begin{equation}
\label{int-last-fw1}
\int_1^t \int \exp(i\phi(\xi, \eta)s) \widehat{f}_w(\xi-\eta, s)|\eta| \widehat{f}_u(\eta, s) \, d\eta \, ds.
\end{equation}
The next lemma plays a crucial role in our analysis.
\begin{lem}\label{lem:cutoff}
Consider $\exp(i \phi s)$, where $\phi$ is given in \eqref{def:phase-low-frequencies}. In the regime $|\eta| > \frac 1 4$, it holds that
\begin{align*}
\mathcal{R}(\exp(i\phi s))&= \exp(- c s), \quad \frac 12 (1-\frac{\sqrt{3}}{2}) \le c:=\mathcal{R}\left(\frac 12 - \frac{1}{2}\sqrt{1-4|\eta|^2}\right) \le  \frac 12.
\end{align*}
\end{lem}
\begin{proof}
it is enough to observe that $\sqrt{1-4|\eta|^2}= i \sqrt{4|\eta|^2-1}$ for $|\eta| \ge \frac 12$. Therefore for $|\eta| \ge \frac 12$ we have $\mathcal{R}(i\phi)=-\frac 12$. It remains to consider $\frac 1 4 < |\eta| < \frac 12$, which provides the lower bound.
\end{proof}
The fundamental consequence is the fact that in the regime $|\eta| > \frac 14$ the operator provides an exponential damping term in time. Therefore in that case there is no need of exploiting the oscillations through the space-time resonance method: it will be enough to apply a crude Cauchy-Schwarz estimate to close the bootstrap. This will be done in Section \ref{sec:highfreq}. In the complementing case where $|\eta| \le \frac 14$ instead, that contribution has the qualitative behavior $- \frac 12 (1-\sqrt{1-4|\eta|^2}) \approx -|\eta|^2$: there is no exponential time decay in this case and a more careful analysis, based on the space-time resonance method \cite{Pusateri, GMS1}, is needed. In order to rigorously implement the previous reasoning, we introduce a cut-off function 
\begin{align}\label{def:cutoff}
\psi( \eta)=\chi(4|\eta|), \quad \text{where} \quad  \chi(|\eta|) \in C_c^\infty, \;  \chi(|\eta|)=\begin{cases}
1 \, \text{for} \, |\eta| \le \frac 12, \\
0 \, \text{for} \, |\eta| \ge 1,
\end{cases}
\end{align}
and we rewrite the above integral
\begin{align}
\label{int-last-fw}
& \int_1^t \int \exp(i\phi(\xi, \eta)s) \widehat{f}_w(\xi-\eta, s)|\eta| \widehat{f}_u(\eta, s) \, d\eta \, ds\notag\\
& = \int_1^t \int \exp(i\phi(\xi, \eta)s) \widehat{f}_w(\xi-\eta, s)|\eta| \widehat{f}_u(\eta, s)\psi(\eta) \, d\eta \, ds\notag\\
&\quad +\int_1^t \int \exp(i\phi(\xi, \eta)s) \widehat{f}_w(\xi-\eta, s)|\eta| \widehat{f}_u(\eta, s)(1-\psi(\eta)) \, d\eta.
\end{align}
In this section we analyze the decay of the first term in the right-hand side. The high-frequency regime in the latter is postponed to the last section.

\paragraph*{Estimate of $\|w(t)\|_{H^N}$.}
\begin{align*}
\int_1^t \|\exp(i\Lambda (t-s)) \, uw \|_{H^N} \, ds 
\le 
\int_1^t \|v(s)\|_{H^N} \|w(s)\|_{H^N} \, ds  
\lesssim
M_0^2(t) \int_1^t s^{-\frac{5}{4}+\varepsilon} \, ds 
\lesssim (1+ t^{-\frac{1}{4}+\varepsilon}) M_0^2(t).
\end{align*}

From Proposition \ref{prop:PS-nonlinear} and formula (\ref{def:Duhamel-formula-w}), it follows that $\|w(t)\|_{H^N} \le t^\varepsilon M_0^2(t)$ for $t \ge 1$.

\paragraph{Estimate of $|w(t)|_{L^\infty}$.}
We use the dispersive properties of the wave propagator $\exp(i\Lambda t)$ in Lemma A.1.
\begin{align*}
\int_1^t |\exp(-i\Lambda(s-t)) vw(s)|_{L_x^\infty}\, ds 
 & \lesssim \int_1^t \frac{1}{(t-s)} ( \|vw(s)\|_{\dot{W}^{2,1}} +  \|\Lambda vw(s)\|_{\dot{W}^{1,1}} ) \, ds \\
& \lesssim  \int_1^t (t-s)^{-1} \|v(s)\|_{H^N} \|w(s)\|_{H^N} \, ds\\
& \lesssim M_0^2(t) \int_1^t (t-s)^{-1}s^{-\frac{5}{4}+\varepsilon} \, ds  \lesssim  t^{-1}M_0^2(t),
\end{align*}

Thus, from (\ref{def:Duhamel-formula-w}) and Proposition \ref{prop:PS-nonlinear}, 
$|w(t)|_{L^\infty} \le t^{-1} M_0^2(t)$.

\paragraph{Estimate of $|R w(t)|_{L^\infty}$.} The Riesz transform $R$ does not play any role on the estimates for the dissipative variables $u, v$.
{In fact, even though the Riesz transform is bounded from $L^p$ to itself only for $1<p<+\infty$, we can use exactly the above strategy 
\begin{align*}
\int_1^t |R \exp(-i\Lambda(s-t)) vw(s)|_{L_x^\infty}\, ds 
 & \lesssim \int_1^t \frac{1}{(t-s)} ( \|R vw(s)\|_{\dot{W}^{2,1}} +  \|\Lambda R vw(s)\|_{\dot{W}^{1,1}} ) \, ds \\
 & \lesssim  \int_1^t (t-s)^{-1} \|Rv(s)\|_{H^N} \|R w(s)\|_{H^N} \, ds\\
 & \lesssim  \int_1^t (t-s)^{-1} \|v(s)\|_{H^N} \|w(s)\|_{H^N} \, ds,
\end{align*}
where in the last inequality we used the boundedness of the Riesz transform in $H^N$. Then, we can proceed as before, so obtaining the same decay of the previous paragraph.}
From Proposition \ref{prop:PS-nonlinear}, one has that $|R w(t)|_{L^\infty} \le t^{-1} M_0^2(t)$.

\paragraph{Estimate of $\|xf_w(t)\|_{L^2}$.}
This is the first point where the space-time resonance method enters the game. As explained just before \eqref{int-last-fw1}, we only need to bound the following integral
\begin{align}
\int_1^t \int \exp(i\phi(\xi, \eta)s) \widehat{f}_w(\xi-\eta, s)|\eta| \widehat{f}_u(\eta, s) \psi(\eta) \, d\eta \, ds,
\end{align}
where $\psi(\eta)$ is in \eqref{def:cutoff}.
Differentiating with respect to $\xi$ we have:
\begin{align}& \int_1^t \int is \partial_\xi \phi  \exp(i\phi s) \widehat{f}_w(\xi-\eta, s) |\eta| \widehat{f}_u(\eta, s)  \psi(\eta) \, d\eta \, ds \quad (i)\notag\\\
& \quad + \int_1^t \int \exp(i\phi s) \partial_\xi \widehat{f}_w(\xi-\eta, s) |\eta| \widehat{f}_u(\eta, s) \psi(\eta) \, d\eta \, ds. \quad (ii)\label{eq:spacenorm}
\end{align}
Let us start with $(ii)$. We have
\begin{align*}
\|(ii)\|_{L^2} & \le \int_1^t \|\exp(-i\Lambda s)x f_w, \, \Lambda \exp(\Delta s){f}_u\|_{L^2} \, ds \lesssim  \int_1^t \| \exp(-i\Lambda s)x{f}_w \|_{L^6} \| \Lambda \exp(\Delta s){f}_u \|_{L^3} \, ds  \\
& \lesssim  \int_1^t  \|x{f}_w(s)\|_{H^1} \|v(s)\|_{L^3} \, ds \lesssim  \int_1^t  \|x{f}_w(s)\|_{H^1} \|v(s)\|_{H^N}  \, ds\lesssim M_0^2(t) \int_1^t s^{\gamma - \frac{5}{4}} \, ds \le  (1+t^{-\frac{1}{4}+\gamma}) M_0^2(t),
\end{align*}
where we used standard Sobolev embeddings and $v \approx \Lambda\exp(\Delta s) f_u$ from (\ref{profile-main-order}). Now we use the space-time resonance method to deal with $(i)$, by integrating by parts in 
 $\eta$. We use the following identity: 
$$\exp(i\phi s)=\dfrac{\partial_\eta \exp(i\phi s) \cdot \de_\eta \phi}{is |\de_\eta \phi|^2}.$$ Therefore, since from \eqref{profile-main-order} it follows that $|\eta| \widehat{f}_u(\eta, s) \approx \exp(|\eta|^2 s) \widehat{v}(\eta, s)$, and (occasionally omitting signs and constants) we obtain
\begin{align*}
(i)&=-\int_1^t \int \de_\xi \phi \cdot \frac{\de_\eta \phi}{|\de_\eta \phi|^2} \exp(i\phi s) \de_\eta (\widehat{f}_w(\xi-\eta, s) |\eta| \widehat{f}_u(\eta, s) \psi(\eta)) \, d\eta \, ds\\
& \quad -\int_1^t \int \de_\eta (\de_\xi \phi \cdot \frac{\de_\eta \phi}{|\de_\eta \phi|^2}) \exp(i\phi s)\widehat{f}_w(\xi-\eta, s) |\eta| \widehat{f}_u(\eta, s) \psi(\eta) \, d\eta \, ds\\
& \approx \int_1^t \int \de_\xi \phi \cdot \frac{\de_\eta \phi}{|\de_\eta \phi|^2} \exp(i\phi s)  \de_\eta (\widehat{f}_w(\xi-\eta, s) \exp(|\eta|^2 s) \widehat{v}(\eta, s)  \psi(\eta)) \, d\eta \, ds\\
& \quad +\int_1^t \int \de_\eta (\de_\xi \phi \cdot \frac{\de_\eta \phi}{|\de_\eta \phi|^2}) \exp(i\phi s)\widehat{f}_w(\xi-\eta, s)  \exp(|\eta|^2 s) \widehat{v}(\eta, s)  \psi(\eta) \, d\eta \, ds.
\end{align*}
We have
\begin{align*}
(i) & \approx \int_1^t \int \de_\xi \phi \cdot \frac{\de_\eta \phi}{|\de_\eta \phi|^2} \exp(i(|\xi|-|\xi-\eta|) s)  \de_\eta \widehat{f}_w(\xi-\eta, s) \widehat{v}(\eta, s)  \psi(\eta) \, d\eta \, ds \quad (i)_1  \\
& \quad +  \int_1^t \int \de_\xi \phi \cdot \frac{\de_\eta \phi}{|\de_\eta \phi|^2}\exp(i(|\xi|-|\xi-\eta|) s)  \widehat{f}_w(\xi-\eta, s) \de_\eta  \widehat{v}(\eta, s)  \psi(\eta) \, d\eta \, ds \quad (i)_2 \\
& \quad +  \int_1^t \int s  \de_\xi \phi \cdot \frac{\de_\eta \phi}{|\de_\eta \phi|^2} \exp(i(|\xi|-|\xi-\eta|) s)  \widehat{f}_w(\xi-\eta, s)  |\eta| \widehat{v}(\eta, s)   \psi(\eta) \, d\eta \, ds \quad (i)_3 \\
& \quad + \int_1^t \int \de_\xi \phi \cdot \frac{\de_\eta \phi}{|\de_\eta \phi|^2} \exp(i(|\xi|-|\xi-\eta|) s) \widehat{f}_w(\xi-\eta, s) \widehat{v}(\eta, s)  \de_\eta \psi(\eta) \, d\eta \, ds \quad (i)_4\\
&\quad + \int_1^t \int \de_{\eta\xi} \phi   \cdot \frac{\de_\eta \phi}{|\de_\eta \phi|^2} \exp(i(|\xi|-|\xi-\eta|) s) \widehat{f}_w(\xi-\eta, s) \widehat{v}(\eta, s)  \psi(\eta) \, d\eta \, ds \quad (i)_5\\
&\quad + \int_1^t \int   \de_\xi \phi   \cdot \frac{\de_{\eta \eta} \phi}{|\de_\eta \phi|^2} \exp(i(|\xi|-|\xi-\eta|) s) \widehat{f}_w(\xi-\eta, s) \widehat{v}(\eta, s)  \psi(\eta) \, d\eta \, ds \quad (i)_6\\
&\quad -2 \int_1^t \int \de_\xi \phi\cdot \frac{\de_\eta \phi \cdot \de_{\eta \eta} \phi }{|\de_\eta \phi|^4} \exp(i(|\xi|-|\xi-\eta|) s) \widehat{f}_w(\xi-\eta, s) \widehat{v}(\eta, s)  \psi(\eta) \, d\eta \, ds. \quad (i)_7.
\end{align*}
We need the following intermediate result.
\begin{lem}\label{lem-crucial}
In the low-frequency regime $|\eta| \le \frac 14$, the following relations hold:
\begin{align}\label{def:symbol-lem}
\frac{\de_\eta \phi}{|\de_\eta \phi|^2} &= {\mu}_0 m_2(\eta)  + \tilde{m}_2(\eta), \\
\de_\xi \phi \cdot \frac{\de_\eta \phi}{|\de_\eta \phi|^2} & = {\mu}^i_0 ({\mu}_0 m_2(\eta)  + \tilde{m}_2(\eta))=\tilde \mu_0^i m_2(\eta) + \mu_0^i \tilde m_2(\eta), \label{def:symbol-lem1}\\
 \de_{\eta\xi} \phi   \cdot \frac{\de_\eta \phi}{|\de_\eta \phi|^2} & = \frac{{\mu}_0^{ii}}{|\xi-\eta|}({\mu}_0 m_2(\eta)  + \tilde{m}_2(\eta))=\frac{\tilde \mu_0^{ii}}{|\xi-\eta|}m_2(\eta) +\frac{\mu_0^{ii}}{|\xi-\eta|}\tilde m_2(\eta), \label{def:symbol-lem2}\\
|\de_{\xi} \phi|^2  \frac{\de_\eta \phi}{|\de_\eta \phi|^2} & = \mu_0^{iii}({\mu}_0 m_2(\eta)  + \tilde{m}_2(\eta))=\tilde{\mu}_0^{iii} m_2(\eta) + \mu_0^{iii} \tilde{m}_2(\eta) , \label{def:symbol-lem21}\\
 \de_{\xi} \phi   \cdot \frac{\de_{\eta\eta} \phi}{|\de_\eta \phi|^2} & = \mu_0^i (\frac{\mu_0^{iv}}{|\xi-\eta|}m_2(\eta)+ m_1(\eta))=\frac{\tilde \mu_0^{iv}}{|\xi-\eta|}m_2(\eta)+\mu_0^i m_1(\eta),  \label{def:symbol-lem3}\\
 | \de_{\xi} \phi|^2\frac{\de_{\eta\eta} \phi}{|\de_\eta \phi|^2} & =  \mu_0^{iii} (\frac{\mu_0^{iv}}{|\xi-\eta|}m_2(\eta)+ m_1(\eta))=\frac{\tilde \mu_0^{v}}{|\xi-\eta|}m_2(\eta)+\mu_0^{iii} m_1(\eta), \label{def:symbol-lem31}\\
   \de_\xi \phi \cdot \de_{\xi \eta} \phi  \frac{\de_{\eta} \phi}{|\de_\eta \phi|^2} & = \mu_0^i (\frac{\tilde{\mu}_0^{ii}}{|\xi-\eta|}m_2(\eta) + \frac{\mu_0^{ii}}{|\xi-\eta|}\tilde{m}_2(\eta)) = \frac{\tilde{\mu}_0^{vi}}{|\xi-\eta|}m_2(\eta) + \frac{\mu_0^{vi}}{|\xi-\eta|}\tilde{m}_2(\eta), \label{def:symbol-lem5}\\
\de_\xi \phi\cdot \frac{\de_\eta \phi \cdot \de_{\eta \eta} \phi }{|\de_\eta \phi|^4} & = \frac{\tilde{\mu}_0^{vii}}{|\xi-\eta|}m_4(\eta) + \mu_0^{vii} m_3(\eta), \label{def:symbol-lem4}\\
|\de_\xi \phi|^2 \frac{\de_\eta \phi \cdot \de_{\eta \eta} \phi }{|\de_\eta \phi|^4} & = \frac{\tilde{\mu}_0^{viii}}{|\xi-\eta|}m_4(\eta) + \mu_0^{viii} m_3(\eta), \label{def:symbol-lem41}
\end{align} 
where $\mu_0=\mu_0(\xi, \eta), \tilde \mu_0=\tilde \mu_0(\xi, \eta)$ (possibly with apex i, ii, $\cdots$) denote any symbol belonging to $\mathcal{B}_0$, while we use the notation $m_k$ for any Fourier multiplier of order $k$. 
\end{lem}
\begin{proof}
We start with \eqref{def:symbol-lem}. We explicitly write down
\begin{align*}
\frac{\de_\eta \phi}{ |\de_\eta \phi|^2}=\frac{\xi-\eta}{|\xi-\eta|} (1-4|\eta|^2) + 2i\eta \sqrt{1-4|\eta|^2},
\end{align*}
where we recall that $|\eta|^2 \le \frac{1}{16}$, so that $1-4|\eta|^2 > \frac 34$. Introducing the notation $\mu_0:= \frac{\xi-\eta}{|\xi-\eta|} \in \mathcal{B}_0$ and noticing that $m_2(\eta)=1-4|\eta|^2, \, \tilde m_2(\eta)=2i\eta \sqrt{1-4|\eta|^2}$ are Fourier multipliers, the proof of \eqref{def:symbol-lem} is over. 
The decomposition  \eqref{def:symbol-lem1} readily follows from Lemma \ref{lemC3}, since $\de_\xi \phi=\frac{\xi}{|\xi|}-\frac{(\xi-\eta)}{|\xi-\eta|} \in \mathcal{B}_0$.
To prove \eqref{def:symbol-lem2}, it is enough to observe that for $j=1,2,3$, 
\begin{align*}
\de_{\eta_j\xi_j} \phi=\de_{\eta_j} \frac{(\eta_j-\xi_j)}{|\eta-\xi|} = \frac{1}{|\xi-\eta|} - \frac{(\xi_j-\eta_j)^2}{|\xi-\eta|^3}. 
\end{align*} 
Since $\de_\xi \phi \in \mathcal{B}_0$, the decomposition \eqref{def:symbol-lem21} follows from \eqref{def:symbol-lem1} and Lemma \ref{lemC3}. 
Besides, \eqref{def:symbol-lem3} follows from Lemma \ref{lemC3} and
\begin{align*}
\frac{\de_{\eta_j\eta_j} \phi}{|\de_\eta \phi|^2}= \left(\frac{-1}{|\xi-\eta|} + \frac{(\xi_j-\eta_j)^2}{|\xi-\eta|^3}\right) (1-4|\eta|^2)  + 2i \sqrt{1-4|\eta|^2} + \frac{i8\eta_j^2}{\sqrt{1-4|\eta|^2}}.
\end{align*}
From \eqref{def:symbol-lem3}, the fact that $\de_\xi \phi\in \mathcal{B}_0$ and Lemma \ref{lemC3} we also obtain \eqref{def:symbol-lem31}.
Next, \eqref{def:symbol-lem5} follows from \eqref{def:symbol-lem2} and Lemma \ref{lemC3}.
Lastly, observe that from \eqref{def:symbol-lem} and \eqref{def:symbol-lem3} we have
\begin{align*}
\de_\xi \phi\cdot \frac{\de_\eta \phi \cdot \de_{\eta \eta} \phi }{|\de_\eta \phi|^4}&=( {\mu}_0 m_2(\eta)  + \tilde{m}_2(\eta)) \left(\frac{\tilde \mu_0^{iv}}{|\xi-\eta|}m_2(\eta)+\mu_0^i m_1(\eta)\right),
\end{align*}
which directly provides \eqref{def:symbol-lem4}, while \eqref{def:symbol-lem41} follows from \eqref{def:symbol-lem} and Lemma \ref{lemC3}. The proof is over.
\end{proof}
Now we use the above lemma to deal with the terms $(i)_1-(i)_7$.
We start with $(i)_1$. Using Lemma \ref{lem-crucial} and Lemma \ref{lem3}, we have
\begin{align}\label{eq:a}
|(i)_1| & \lesssim \int_1^t \| x f_w\|_{L^{6}} \|v \|_{W^{2,3 }} \lesssim M_0(t)^2 \int_1^t s^{\gamma-\frac 54} \, ds \lesssim t^\gamma M_0^2(t). 
\end{align}
Now we deal with $(i)_2$. We use the dispersive estimate \eqref{est:app-Lp} with $p=6-\delta$, so that
\begin{align}\label{eq:b}
|(i)_2|& \lesssim  \int_1^t \|e^{i\Lambda s} f_w\|_{L^{6-\delta}} \|x v\|_{W^{2, 3+\delta}} \, ds \lesssim \int_1^t s^{-\frac 23 +\delta} \|\Lambda^{\frac 43+\delta} f_w\|_{L^{\frac 65+\delta}} \|xv\|_{H^N} \, ds\notag \\
& \lesssim \int_1^t s^{-\frac 23 +\delta} \|x f_w\|_{H^2} \|xv\|_{H^N} \, ds  \lesssim  M_0(t)^2 \int_1^t s^{\gamma+\delta - \frac 23 - \frac 34} \, ds \lesssim  t^\gamma M_0^2(t).
\end{align}
Consider the term $(i)_3$. Using again the dispersive estimate and using also that from \eqref{PK_Decay_estimates} we have $\||\eta| \widehat{v}\|_{L^2} \approx t^{-\frac 74}M_0(t)$, we get
\begin{align*}
|(i)_3|& \lesssim \int_1^t  s\|e^{i\Lambda s} f_w\|_{L^{6-\delta}} \||\eta| \widehat{v}\|_{W^{2, 3+\delta}} ds  \lesssim \int_1^t s^{1-\frac 2 3 +\delta}  \| \Lambda^{\frac 43+\delta} f_w\|_{L^{\frac 65+\delta}} \||\eta| \widehat{v}\|_{H^N} \, ds \\
& \lesssim \int_1^t s^{\frac 13 + \delta} \|x f_w\|_{H^2} \||\eta| \widehat{v}\|_{H^N} \, ds  \lesssim M_0(t)^2 \int_1^t s^{\frac 13 - \frac 74 + \delta + \gamma} \, ds  \lesssim t^\gamma M_0(t)^2.
\end{align*}
The term $(i)_4$ is also analogous since the cut-off does not play any role and its derivatives are uniformly bounded. 
We now turn to $(i)_5$. We have using \eqref{ineq:frac-int1} that 
\begin{align}
|(i)_5|& \lesssim \int_1^t \|\Lambda^{-1} f_w\|_{L^6} \|v\|_{W{2,6}} + \| f_w\|_{L^3} \|\Lambda^{-1}v\|_{W^{2,6}} \, ds\notag \\
& \lesssim \int_1^t \|w\|_{H^N} \|v\|_{H^N} \, ds \lesssim M_0(t)^2 \int_1^t s^\eps s^{-\frac 54} \, ds \lesssim t^\gamma M_0(t)^2. \label{est:i5}
\end{align}
The terms $(i)_6-(i)_7$ are really similar to the previous one, therefore we omit them.
 Hereafter $\mu_0=\mu_0(\xi, \eta), \tilde \mu_0=\tilde \mu_0(\xi, \eta)$ (possibly with apex i, ii, $\cdots$) denote any symbol belonging to $\mathcal{B}_0$.
\paragraph{Estimate of $\|\Lambda x f_w(t) \|_{H^1}$.}
It consists in estimating the terms of the previous paragraph with all the terms of the integrand multiplied by $|\xi|$. Those estimates are really similar and therefore we omit them. 

\paragraph{Estimate of $\||x|^2\Lambda f_w(t) \|_{H^1}$.}
We start by recalling that
\begin{align*}
|\xi| \widehat{f}_w = |\xi| \widehat{w}_0(\xi) & + \int_1^t  \exp(i|\xi| s) |\xi| \widehat{T}_{m(\xi, \eta)} (\exp(-i|\xi-\eta| s) \widehat{f}_w, \exp(-i|\eta| s) \widehat{f}_w){\psi(\eta)}   \,  ds\\
& +  \int_1^t \int |\xi| \exp(i\phi s) \widehat{f}_w |\eta| \widehat{f}_u {\psi(\eta)}\, d\eta \, ds.
\end{align*}
Differentiating with respect to $\xi$ (and occasionally omitting signs and constants) one has:
\begin{align*}
\partial_\xi (|\xi| \widehat{f}_w ) & \approx \frac{\xi}{|\xi|} \widehat{w}_0 + |\xi| \partial_\xi \widehat{w}_0 + \int_1^t \partial_\xi(\exp(i\Lambda s) |\xi| \widehat{T}_{m(\xi, \eta)} (\exp(-i\Lambda s) \widehat{f}_w, \exp(-i \Lambda s) \widehat{f}_w)){\psi(\eta)} \, ds  \\
& \quad + \int_1^t \int \frac{\xi}{|\xi|} \exp(i\phi s) \widehat{f}_w |\eta| \widehat{f}_u  {\psi(\eta)}\, d\eta \, ds  + \int_1^t \int |\xi| is \partial_\xi \phi  \exp(i\phi s)|\eta| \widehat{f}_w \widehat{f}_u {\psi(\eta)} \, d\eta \, ds\\
&\quad + \int_1^t \int |\xi|  \exp(i\phi s) |\eta| \partial_\xi\widehat{f}_w \widehat{f}_u \, {\psi(\eta)} \, d\eta  \, ds.
\end{align*}
Differentiating again in $\xi$,
\begin{align*}
\partial_{\xi}^2&(|\xi|\widehat{f}_w)=\dfrac{\xi^2}{|\xi|^3}\widehat{w}_0+ 2\dfrac{\xi}{|\xi|}\partial_\xi \widehat{w}_0 + |\xi| \partial^2_\xi \widehat{w}_0 + \int_1^t \partial^2_\xi(\exp(i\Lambda s) |\xi| \widehat{T}_{m(\xi, \eta) } (\exp(-i\Lambda s) \widehat{f}_w, \exp(-i\Lambda s) \widehat{f}_w)) {\psi(\eta)} \, ds  \\
&\quad+ \int_1^t \int s^2 \exp(i\phi s) |\partial_\xi\phi|^2 |\eta| |\xi| \widehat{f}_w \widehat{f}_u {\psi(\eta)}  \, d\eta \, ds  \quad (a) + \int_1^t \int s \exp(i\phi s) \partial_\xi^2 \phi |\xi| |\eta| \widehat{f}_w \widehat{f}_u {\psi(\eta)} \, d\eta \, ds \quad (b) \\
&\quad+ \int_1^t \int \exp(i \phi s) \dfrac{\xi}{|\xi|} \partial_\xi\widehat{f}_w |\eta|  \widehat{f}_u { \psi(\eta)} \, d\eta \, ds \quad (c) + \int_1^t \int s \exp(i \phi s) |\xi| |\eta| \partial_\xi \phi \partial_\xi \widehat{f}_w \widehat{f}_u { \psi(\eta)} \, d\eta \, ds \quad (d) \\
&\quad+ \int_1^t \int is \exp(i\phi s) \partial_\xi \phi |\eta| \dfrac{\xi}{|\xi|} \widehat{f}_w \widehat{f}_u {\psi(\eta)} \, d\eta \, ds \quad (e) + \int_1^t \int \exp(i \phi s) \dfrac{\xi^2}{|\xi|^3} \widehat{f}_w |\eta| \widehat{f}_u { \psi(\eta)} \, d\eta \, ds \quad (f)\\
&\quad+ \int_1^t \int \exp(i\phi s)|\xi| |\eta| \partial_\xi^2 \widehat{f}_w \widehat{f}_u  {\psi(\eta)}\, d\eta \, ds \quad  (g).
\end{align*}

The first integral is studied in \cite{Pusateri}. As recalled in Proposition \ref{prop:PS-nonlinear}, it grows like $t$.
Let us start from (a). We integrate it by parts in $\eta$, so that
\begin{align*}
(a)&\approx \int_1^t \int s \exp(i(|\xi|-|\xi-\eta|)s) |\de_\xi \phi|^2 \frac{\de_\eta \phi}{|\de_\eta \phi|^2} |\xi| \de_\eta \widehat{f}_w \widehat{v} \psi \, d\eta \, ds \quad (a)_1 \\
& \quad + \int_1^t \int s \exp(i(|\xi|-|\xi-\eta|)s) |\de_\xi \phi|^2 \frac{\de_\eta \phi}{|\de_\eta \phi|^2} |\xi| \widehat{f}_w \de_\eta \widehat{v} \psi \, d\eta \, ds \quad (a)_2 \\
& \quad + \int_1^t \int s^2 \exp(i(|\xi|-|\xi-\eta|)s) |\de_\xi \phi|^2 \frac{\de_\eta \phi}{|\de_\eta \phi|^2} |\xi| \widehat{f}_w |\eta| \widehat{v} \psi \, d\eta \, ds \quad (a)_3 \\
&\quad + \int_1^t \int s \exp(i(|\xi|-|\xi-\eta|)s) |\de_\xi \phi|^2 \frac{\de_\eta \phi}{|\de_\eta \phi|^2} |\xi| \widehat{f}_w \widehat{v} \de_\eta\psi \, d\eta \, ds \quad (a)_4\\
& \quad + \int_1^t \int s \exp(i(|\xi|-|\xi-\eta|)s) |\de_\xi \phi|^2 \frac{\de_{\eta\eta} \phi}{|\de_\eta \phi|^2} |\xi| \widehat{f}_w \widehat{v} \psi \, d\eta \, ds \quad (a)_5\\
& \quad + 2 \int_1^t \int s \exp(i(|\xi|-|\xi-\eta|)s) \de_\xi \phi \cdot \de_{\xi \eta} \phi  \frac{\de_{\eta} \phi}{|\de_\eta \phi|^2} |\xi| \widehat{f}_w \widehat{v} \psi \, d\eta \, ds \quad (a)_6\\
& \quad -2 \int_1^t \int s \exp(i(|\xi|-|\xi-\eta|)s)  |\de_\xi \phi|^2 \frac{\de_\eta \phi \cdot \de_{\eta \eta} \phi}{|\de_\eta \phi|^2} |\xi| \widehat{f}_w \widehat{v} \psi \, d\eta \, ds \quad (a)_7.\\
\end{align*}
We proceed with our computations, using the dispersive estimate \eqref{est3-lem1} and Lemma \ref{lem-crucial}. From \eqref{def:symbol-lem21} have
\begin{align*}
|(a)_1|& \lesssim \int_1^t s \|x f_w\|_{W^{1, 2 }} \|v\|_{H^N} \, ds \lesssim M_0^2(t) \int_1^t s^{\gamma+1-\frac 54} \, ds  \lesssim t M_0^2(t). 
\end{align*}
The bounds for $(a)_2-(a)_4$ are similar to \eqref{eq:b} and we omit them.
Consider $(a)_3$. We use again \eqref{est:app-Lp} with $p=(\frac 38-\frac \gamma 2)^{-1}, p'=(\frac 18 + \frac \gamma 2)^{-1}$, so that we have
\begin{align*}
|(a)_3| & \lesssim \int_1^t s^2 \|e^{i\Lambda s} f_w\|_{W^{1, (\frac 38-\frac \gamma 2)^{-1}}} \||\eta| \widehat{v}\|_{H^N} \, ds \\
& \lesssim \int_1^t s^{2-\frac 1 4 - \gamma }  \|\Lambda^{\frac 12+ 2 \gamma }  f_w\|_{W^{1, (\frac 18 + \frac \gamma 2)^{-1}}}  \||\eta| \widehat{v}\|_{H^{N}} \, ds \\
& \lesssim \int_1^t s^{2-\frac 1 4 - \gamma }  \| x  f_w\|_{H^2}  \||\eta| \widehat{v}\|_{H^{N}} \, ds \\
& \lesssim M_0(t)^2 \int_1^t s^{2-\frac 14-\gamma - \frac 74 + \gamma} \, ds \lesssim t M_0^2(t).
\end{align*}
Using Lemma \ref{lem-crucial}, the last two terms $(a)_5-(a)_7$ are really analogous to $(i)_5$ in \eqref{est:i5}. In particular the symbol in $(a)_5$ is treated in \eqref{def:symbol-lem31} and the symbol in $(a)_6$ in \eqref{def:symbol-lem5} and the symbol in $(a)_7$ is studied in \eqref{def:symbol-lem41}.
Now consider (b), which gives
\begin{align*}
\partial^2_\xi \phi= \frac{\mu_0}{|\xi|} + \frac{\tilde \mu_0}{ |\xi-\eta|}.
\end{align*}
Therefore 
\begin{align*}
\partial^2_\xi \phi |\xi-\eta||\xi| = \mu_0 |\xi-\eta| + \tilde \mu_0 |\xi| = \mu_1 \in \mathcal{B}_1. 
\end{align*} 
Using this observation in the estimate for $(b)$, Lemma \ref{lem3} and \eqref{lemC3} with $\alpha=1, q=6, p=2$, we have:
\begin{align*}
\|(b)\|_{L^2}  & \lesssim \int_1^t s \|T_{\mu_1} (\exp(i\Lambda s)\Lambda^{-1}\exp(-i\Lambda s) f_w,  \exp(\Delta s) \Lambda f_u)\|_{L^2} \, ds\\
& \lesssim  \int_1^t s \|\Lambda^{-1} \exp(-i\Lambda s) f_w\|_{W^{1,6}}  \|\exp(\Delta s) \Lambda f_u)\|_{W^{1,3}} \, ds \\
&\quad  +  \int_1^t s \| \exp(-i\Lambda s) f_w\|_{W^{1,3}}  \|\Lambda^{-1} \exp(\Delta s) \Lambda f_u)\|_{W^{1,6}} \, ds\\
& \lesssim  M_0^2(t)  \int_1^t s^{1+\gamma-\frac{5}{4}} \, ds  \lesssim t^{\frac{3}{4}+\gamma}M_0^2(t) .
\end{align*}

Note that (c) and (e) are similar. Let us consider (d). The symbol $$\mu_1(\xi, \eta)=|\xi| \partial_\xi \phi  \in \mathcal{B}_1.$$
\begin{align*}
\|(d)\|_{L^2} &\lesssim \int_1^t s \|\exp(i\Lambda s) T_{\mu_1}( \exp(-i \Lambda s) xf_w, \exp(\Delta s) \Lambda f_u)\|_{L^2} \, ds \\
&\lesssim \int_1^t s \|\exp(-i\Lambda s) x f_w\|_{W^{1, 6}} \|\exp(\Delta s) \Lambda f_u)\|_{W^{1,3}} \, ds \\
&\lesssim  M_0^2(t) \int_1^t s^{1+\gamma-\frac{5}{4}} \,ds \lesssim   s^{\frac{3}{4}+\gamma} M_0^2(t).
\end{align*}

We estimate (f). Notice that $\frac{\xi^2}{|\xi|^3} = \frac{\mu_0}{|\xi|}$. Therefore, using \eqref{lemC3} with $\alpha=1, q=2, p=\frac65$, we have
\begin{align*}
\|(f)\|_{L^2} &\lesssim \int_1^t s \|\Lambda^{-1} \exp(i\Lambda s) T_{\mu_0}( \exp(-i\Lambda s) f_w, \exp(\Delta s) \Lambda f_u)\|_{L^2} \, ds \\
&\lesssim \int_1^t s \|\exp(i\Lambda s) T_{\mu_0}( \exp(-i\Lambda s) f_w, \exp(\Delta s) \Lambda f_u) \|_{L^\frac{6}{5}} \, ds \\
&\lesssim \int_1^t s \|\exp(-i\Lambda s) f_w \|_{L^3} \| \exp(\Delta s) \Lambda f_u\|_{L^2} \, ds \\
&\lesssim  M_0^2(t) \int_1^t s^{1-\frac{1}{3}+\gamma-\frac{5}{4}} \, ds \lesssim t^{\frac{5}{12}+\gamma} M_0^2(t),
\end{align*}
where we used Lemma \ref{lem2}, Lemma \ref{lem3} and Lemma \ref{lem1}. The last one is (g).
\begin{align*}
\|(g)\|_{L^2} & \lesssim \int_1^t \|\exp(i \Lambda s)\Lambda   \exp(-i \Lambda s) x^2 f_w, \exp(\Delta s) \Lambda f_u\|_{L^2} \, ds \\
&\lesssim  \int_1^t \|x^2 \Lambda  f_w\|_{L^{6}} \|\exp(\Delta s) \Lambda f_u\|_{L^3} \, ds \\
&\lesssim  M_0(t) \int_1^t \| \Lambda x^2 f_w\|_{H^1} s^{-\frac{5}{4}} \, ds\\
& \lesssim  M_0^2(t) \int_1^t s^{1-\frac{5}{4}} \, ds \lesssim t^\frac{3}{4} M_0^2(t).
\end{align*}

\subsection{The high frequency regime}\label{sec:highfreq}
This is the regime where $|\eta|$ is at high frequencies: more precisely, we are dealing with the estimates of the second addend in \eqref{int-last-fw} where $|\eta| > \frac 14$.

As already widely mentioned, the only difficult term for which we need a careful and explicit estimate is the integral \eqref{int-last-fw}.
We recall once more the expression of the time phase $\phi(\xi, \eta)$ in \eqref{def:phase-low-frequencies},
\begin{align*}
i\phi(\xi, \eta)=i(|\xi|-|\xi-\eta|) - \frac 12 + \frac 12 \sqrt{1-4|\eta|^2}.
\end{align*}
From Lemma \ref{lem:cutoff} we have that $\mathcal{R}(i\phi)=-c$, where $\mathcal{R}$ denotes the real part. Being $\phi$ the time phase, this implies that the term $\exp (i\phi(\eta, \xi) s)$ in the integrand of \eqref{int-last-fw} does not only have an oscillating contribution, but also an exponentially decaying one, of the order of $\exp(-c s)$. This analysis largely simplifies the estimates of the high-frequency regime. 
We remark that an exponential decay in time should in fact be expected from the linear study of the dissipative part of the system \eqref{PK_Green_function_decomposition}, where we showed that in the high frequency regime $|\eta| > a$, we have
\begin{align*}
\begin{pmatrix}
\widehat{u}\\
\widehat{v}
\end{pmatrix}=
\begin{pmatrix}
O(|\xi|^2) & |\xi| \\
|\xi| & 1+O(|\xi|^2)
\end{pmatrix}
\begin{pmatrix}
\widehat{f}_u\\
\widehat{f}_v
\end{pmatrix}
\exp (-ct),
\end{align*}
where now $c$ is a generic positive constant. 
Thus we explicitly write down the reasoning for the most difficult bound of the previous section, i.e. $\|x f_w\|_L^2$, in the high-frequency regime, in order to show that it is readily obtained from the previous observations. Being simpler, the other estimates for this high frequency terms are omitted.
Therefore we differentiate \eqref{int-last-fw} again with respect to $\xi$ as done in \eqref{eq:spacenorm} and we focus on the term (i), for which we recall that the decay of the linear semigroup was not enough to close the bootstrap argument and in the low-frequency regime we were forced to exploit the oscillations by means of integration by parts in space. 
Thanks to the exponential decay, in the present high-frequency regime it is enough to apply a crude Cauchy-Schwarz inequality. 
{
More precisely, let us consider the quantity
\begin{align*}
& \|v(t)\|_{H^N}^{\rm{h}}, \quad \text{where, for any} \quad f,   \quad \text{we define} \quad 
 \|f\|_{H^N}^{\rm{h}}:=\|(1+\Lambda(\xi))^\frac N 2 \widehat{f}(\xi)(1-\psi(\xi))\|_{L^2},
\end{align*}
where $\psi(\xi)$ is given by \eqref{def:cutoff}(i.e. a cut-off at high-frequency).
It is immediate to deduce from Proposition \ref{prop:decay-SK}, point v),  that the following high-frequency estimate holds for the unknown $v$ of system \eqref{PK_system_w}: 
\begin{align*}
\|v(t)\|_{H^N}^{\rm{h}} & \le C E_N \exp(-ct) + C \int_1^t \exp(-c(t-s)) \|v(s)\|_{L^2}^2 \, ds \\
& \le  C E_N \exp(-ct) + CM_0^2(t) \int_1^t  \exp(-c(t-s)) s^{-\frac 52} \, ds \\
& \lesssim   E_N  \exp(-ct) + t^{-\frac 52} {M}_0^2(t),
\end{align*}
so that 
\begin{align}\label{est:highfreq}
t^\frac 52 \|v(t)\|_{H^N}^{\rm{h}} \lesssim  E_N + {M}_0^2(t).
\end{align}
Now we use this information to estimate the (high-frequency) term (i) in \eqref{eq:spacenorm}, which we rewrite below
\begin{align*}
& \int_1^t \int is \partial_\xi \phi  \exp(i\phi s) \widehat{f}_w(\xi-\eta, s) |\eta| \widehat{f}_u(\eta, s)  (1-\psi(\eta)) \, d\eta \, ds.
\end{align*}
The task is to estimate this term in $L^2$, with an upper bound of the order of $t^\gamma$, according to the bootstrap estimate in \eqref{X_norm}. 
First, we can apply Lemma \ref{lem3} to $\mu_0(\xi, \eta)=\de_\xi  \phi(\xi, \eta)$, which yields 
\begin{align*}
\|(i)\|_{L^2} & \lesssim  \int_1^t   s\|T_{\mu_0 (1-\psi) } (\widehat{w}(\xi-\eta, s), \widehat{v}(\eta, s))\|_{L^2} \, ds \, {\lesssim \int_1^t s  \|w(s)\|_{H^N} \|v(s)\|^{\text{h}}_{H^N} \, ds} \\
& \lesssim  M_0(t) (E_N+{M}_0^2(t)) \int_1^t  s^{1+\eps} \times s^{-\frac 52}   \, ds \lesssim M_0(t) (E_N+{M}_0^2(t)),
\end{align*}
where the last inequalities follow by using \eqref{est:highfreq} and the definition of $M_0(t)$ in \eqref{X_norm}.
Finally, the argument is concluded, choosing $\eps$ small enough, by the following inequalities:
\begin{align*}
M_0(t) \lesssim E_N + t^\frac 52 \|v(t)\|^{\rm{h}}_{H^N}  M_0(t) + M_0(t)^2 \lesssim  E_N (1+M_0(t)) + M_0^2(t)\lesssim \eps (1+M_0(t)) + M_0^2(t).
\end{align*} 
 }

\appendix 
\section{Toolbox}
First, we collect two estimates on the wave propagator $\exp(i\Lambda t)$.

\begin{Lem}[\cite{Pusateri}]\label{lem1}
The following estimates on the wave propagator $\exp(i\Lambda t)$ hold:
\begin{align}
\label{est:app-infty}
|\exp(i\Lambda t)f|_{L^\infty} & \lesssim t^{-1} [\|f\|_{\dot{W}^{2,1}}+\|\Lambda f\|_{\dot{W}^{1,1}}]; \\
\label{est:app-Lp}
\|\exp(i\Lambda t)f\|_{L^p} & \lesssim  t^{-1+\frac{2}{p}} \|\Lambda^{2-\frac{4}{p}}f\|_{L^{p'}}, \quad \text{for  }  2 \le p < \infty.
\end{align}
We also have  the following bounds:
\begin{align}
\label{est3-lem1}
\|w\|_{W^{1,p}} & \lesssim  t^{-1+\frac{2}{p}}M_0(t), \quad \text{for  } 2 \le p \le 4;\\
\label{est4-lem1}
\|\Lambda^{-1} w\|_{W^{1,p}} & \lesssim  t^{\gamma-1+\frac{2}{p}}M_0(t), \quad \text{for  } 4 \le p < 6.
\end{align}
\end{Lem}

We recall two estimates on fractional integration, see \cite{Pusateri, GMS1} for more details.

\begin{Lem}[\cite{Pusateri}]\label{lem2} Let $\Lambda=(-\Delta)^\frac 12$. Then, for $0<\alpha < \frac{3}{p}$,
\begin{align}
\label{ineq:frac-int1}
\Bigg\|\frac{1}{\Lambda^\alpha} f\Bigg\|_{L^q} & \lesssim  \|f\|_{L^p} \quad \text{for    } 1<p,q<\infty \; \text{and } \alpha=\frac{3}{p}-\frac{3}{q};\\
\label{ineq:frac-int2}
\Bigg\|\frac{1}{\Lambda^\alpha} \exp(i\Lambda t) f\Bigg\|_{L^q} & \lesssim \|f\|_{L^p} \quad \text{for    } 1<p \le 2 \le q<\infty \; \text{and } \alpha=\frac{3}{p}-\frac{3}{q}.
\end{align}
\end{Lem}

We conclude this appendix section with the generalized version of the H{\"o}lder inequality for Coifman-Meyer operators and symbols in $\mathcal{B}_s$.

\begin{Lem}[\cite{CM}]\label{lemCF}
if a symbol $m(\xi, \eta)$ belongs to the Coifman-Meyer class as in Definition \ref{def:CM}, then, for $\frac{1}{r}=\frac{1}{p}+\frac{1}{q}, \; 1<p,q \le \infty, \, 0 < r < \infty$.
$$\|T_m (f, g)\|_{L^r} \lesssim \|f\|_{L^p} \|g\|_{L^q}.$$
\end{Lem}

\begin{Lem}[\cite{Pusateri}]\label{lem3}
Let $p, q, r$ be such that $\frac{1}{r}=\frac{1}{p}+\frac{1}{q}$ and $1<p, q, r < \infty$. 
If $m(\xi, \eta) \in \mathcal{B}_0$ as in Definition \ref{def:symbols}, then:
$$\|T_m (f, g)\|_{L^r} \lesssim \|f\|_{L^p} \|g\|_{L^q}.$$
If $m(\xi, \eta) \in \mathcal{B}_s, \, s>0$ as in Definition \ref{def:symbols}, then:
$$\|\Lambda^k T_m (f, g)\|_{L^r} \lesssim \|f\|_{ W^{s+k, p}} \|g\|_{L^q} + \|f\|_{L^p} \|g\|_{ W^{s+k, q}}.$$
\end{Lem}

\begin{lem}[Lemma C.3, \cite{GMS1}]\label{lemC3}
If $\mu_s(\xi, \eta) \in \mathcal{B}_s, \mu_{s'}(\xi, \eta) \in \mathcal{B}_{s'}$, then $\mu_s \mu_{s'} \in \mathcal{B}_{s+s'}$.
\end{lem}

\section*{Acknowledgements}
We are grateful to Fabio Pusateri for reading carefully a first version of this manuscript. RB thanks David Lannes for some helpful discussions. RN thanks Bernard Hanouzet who constantly inspired his research for more than 35 years: this paper is just another fruit of this influence. 

\section*{Conflict of interest}
The authors declare that they have no conflict of interest.

\bibliographystyle{plain}

\end{document}